\documentclass[11pt, reqno]{amsart}

\usepackage{graphicx}
\usepackage{amsmath,amssymb,mathtools,amsopn,amsfonts,amsthm}
\usepackage{bbm}
\usepackage{subfigure}
\usepackage{eepic}
\usepackage{tikz}
\usetikzlibrary{calc,angles,positioning,intersections,quotes}
\usepackage{a4wide}
\usepackage[utf8]{inputenc}
\usepackage{epsfig}
\usepackage{latexsym}
\usepackage[shortlabels]{enumitem}
\usepackage[UKenglish]{babel}
\usepackage{verbatim}
\usepackage{mathrsfs}

\usepackage[initials]{amsrefs}
\usepackage{hyperref}
\usepackage{float}
\usepackage{enumitem}

\numberwithin{equation}{section}
\theoremstyle{plain}

\newtheorem{thm}{Theorem}[section]
\newtheorem{lma}[thm]{Lemma}

\newtheorem{cor}[thm]{Corollary}
\newtheorem{prop}[thm]{Proposition}
\theoremstyle{definition}
\newtheorem{defn}[thm]{Definition}

\newtheorem{conj}[thm]{Conjecture}

\newtheorem{ques}[thm]{Question}

\newtheorem{eg}[thm]{Example}

\def\R{\ensuremath{\mathbb R}}
\def\N{\ensuremath{\mathbb N}}
\def\Q{\ensuremath{\mathcal Q}}

\def\U{\ensuremath{\mathcal U}}
\def\B{\ensuremath{\mathcal B}}

\def\C{\ensuremath{\mathcal C}}
\def\L{\ensuremath{\mathcal L}}

\def\P{\ensuremath{\mathcal P}}

\def\F{\ensuremath{\mathcal F}}

\def\A{\ensuremath{\mathscr A}}

\def\n{\ensuremath{n}}

\def\E{\mathbb E}

\def\sm{\setminus}

\renewcommand{\i}{\mathtt{i}}
\renewcommand{\j}{\mathtt{j}}
\renewcommand{\k}{\mathtt{k}}

\renewcommand{\H}{\mathcal{H}}
\renewcommand{\A}{\mathcal{A}}
\renewcommand{\C}{\mathcal{C}}
\renewcommand{\U}{\mathcal{U}}
\renewcommand{\F}{\mathcal{F}}

\newcommand{\clabel}[1]{\makebox[\labelwidth][c]{#1}}

\newtheoremstyle{nonumberplain}
  {}{}%
  {\itshape}{}%
  {\bfseries}{.}%
  {.5em}{}

\theoremstyle{nonumberplain}

\newtheorem*{assumptions}{Standing assumptions}

\title{Dynamical blanket times}
\date{\today}


\author[N. Jurga]{Natalia Jurga}\address{Natalia Jurga\\ Mathematical Institute\\
University of St Andrews\\
North Haugh\\
St Andrews\\
KY16 9SS\\
Scotland \\}
\email{naj1@st-andrews.ac.uk}
\urladdr{https://www.nataliajurga.co.uk/}

\author[M. Todd]{Mike Todd}
\address{Mike Todd\\ Mathematical Institute\\
University of St Andrews\\
North Haugh\\
St Andrews\\
KY16 9SS\\
Scotland \\} \email{{m.todd@st-andrews.ac.uk }}
\urladdr{https://mtoddm.github.io/}

\thanks{NJ was partially supported by a Leverhulme Early Career Fellowship (ECF-2021-385). MT was partially supported by the FCT (Funda\c{c}ão para a Ciência e a Tecnologia) 
project 2022.07167.PTDC and by his EPSRC grant UKRI1120. Thanks to Stephen Cantrell for reading an early draft of this paper and for useful discussions around the proof of Lemma \ref{lem:weakreg}. }

\begin{document}

 \maketitle
\begin{abstract}
We introduce \emph{dynamical blanket times}, which quantify how quickly the empirical distribution along a typical orbit approximates the invariant measure. These can be viewed as a measure-theoretic analogue of the previously introduced \emph{dynamical cover time}, which measures how quickly an orbit becomes dense in the space. Motivated by analogous comparability questions for random walks on graphs, we investigate how much longer it takes for a dynamical system to ``blanket'' than to ``cover''.

For finite-branch, uniformly expanding interval maps, we obtain upper bounds on the expected blanket time in terms of the spatial scale and the precision of approximation. In the special case where the invariant measure is absolutely continuous with respect to Lebesgue, this yields comparability between the expected blanket and cover times, uniformly across all sufficiently small scales.

Our approach combines two main ingredients. First, we establish large deviation estimates for hitting times which are uniform over both the target location and the spatial scale. Second, using methods from multifractal analysis, we construct a finite discretisation of the invariant measure which reduces the problem to a suitably controlled discrete model.

\end{abstract}

\section{Introduction}

Let $f:X \to X$ be a map on a metric space $X$, and let $\mu$ be a fully supported Borel probability measure on $X$. 
If the dynamical system $(f,X,\mu)$ is ergodic, then ``most'' of its orbits exhibit two fundamental properties: for $\mu$-almost every $x \in X$,
\begin{itemize}[
    labelwidth=3.2cm,   
    labelsep=0.5cm,     
    leftmargin=3.7cm    
]
  \item[\clabel{(covering)}]
    the orbit $\{f^n(x)\}_{n\in\mathbb{N}}$ is dense in $X$ and
  \item[\clabel{(equidistribution)}]
    the sequence of empirical measures
    $\left\{\frac{1}{n}\sum_{i=0}^{n-1}\delta_{f^i(x)}\right\}$
    converges weakly to $\mu$.
\end{itemize}

We refer to these as \emph{local statistical properties}, in the sense that they concern the behaviour of a single typical orbit. This contrasts with \emph{global statistical properties}, such as decay of correlations, limit theorems, or spectral gaps for transfer operators, which describe averaged or spectral features of the system as a whole. Global statistical properties have been studied intensively for several decades; by comparison, quantitative questions about \emph{local} orbit-level behaviour remain far less developed. One of our aims is to highlight the richness of this local viewpoint.

In this paper we are interested in \emph{quantifying the rate} at which covering and equidistribution occur, and in comparing these two rates. Given a single orbit, we ask:

\begin{itemize}
    \item How long does it take for the orbit to become $\delta$-dense in the space?
    \item How long does it take before the empirical distribution along the orbit is close to $\mu$, at the same spatial scale $\delta$?
\end{itemize}

To make these questions precise we introduce two \emph{stopping times} that record how a single orbit explores the space. Stopping times are a natural tool for orbit-level phenomena in dynamics: the classical example is the hitting or return time to a set, which plays a central role in shrinking target problems, Diophantine approximation, and inducing schemes for non-uniformly hyperbolic systems. Here, we introduce stopping times tailored to covering and equidistribution.

For $\delta \in (0,1)$ we define the \emph{cover time at scale $\delta$},
\[
  \C_\delta: X \to \mathbb{N},
\]
by
\[
  \C_\delta(x):=\inf\bigl\{n \ge 1: 
      \{x,f(x), \ldots , f^{n-1}(x)\} \text{ is $\delta$-dense in $X$}
  \bigr\}.
\]
For $\epsilon\in (0,1)$ we define the \emph{$\epsilon$-approximate blanket time at scale $\delta$},
\[
  \B_\delta^{1-\epsilon}: X \to \mathbb{N},
\]
by
\[
  \B_\delta^{1-\epsilon}(x):=\min\left\{N \ge 1: 
      \forall y \in \mathrm{supp}\,\mu,\ 
      \frac{1}{N} \sum_{i=0}^{N-1} 
        \mathbf{1}_{B(y,\delta)}(f^i x) 
      \geq (1-\epsilon)\,\mu\bigl(B(y,\delta)\bigr)
  \right\}.
\]

The cover time measures how quickly a typical orbit fills the space at scale $\delta$. The blanket time measures how quickly the empirical distribution along a typical orbit approximates the invariant measure $\mu$ at the same scale. Both are measurable functions on $(X,\mu)$, and the quantities of interest are their expectations and their asymptotic behaviour as $\delta\to0$ and $\epsilon\to0$.

Clearly an orbit must first be $\delta$-dense before it can have any hope of being $(1-\epsilon)$-blanketing at scale $\delta$, so we always have
\[
  \mathbb{E}_\mu\C_\delta \;\leq\;
  \mathbb{E}_\mu \B_\delta^{1-\epsilon}.
\]
Our main question is: \emph{how much longer} do we expect it to take to blanket than to cover?

After a brief discussion of the probabilistic background, we formulate this more precisely.

\subsection{Historical context and motivation}

Orbit-level stopping time questions of this type have been studied extensively in probability theory and theoretical computer science, but comparatively little in ergodic theory. The main exception is the case of hitting and return times, which feature prominently in shrinking target problems, Diophantine approximation, and in constructing inducing schemes for slowly mixing systems. By contrast, quantitative questions such as cover and blanket times -- which are central in the probabilistic literature -- have barely been explored for deterministic dynamical systems (see, however, results in the line of \cite{DehTaq89} for related work considering the convergence of the cumulative distribution function of empirical measures). This illustrates a broader phenomenon: while global statistical properties in ergodic theory have been deeply understood and refined, the local, orbit-by-orbit behaviour remains far less developed.

The cover time was introduced only recently in the dynamical setting in \cite{JurTod25}, motivated in part by analogous notions in the ``chaos game'' and random walks on graphs. To the best of our knowledge, the blanket time has not previously been studied in a deterministic dynamical context. One can view $\B_{\delta}^{1-\epsilon}$ as a quantitative, orbit-level measure of equidistribution, complementing the recent surge of results on \emph{effective} (i.e.\ global, spectral) equidistribution.

Our work is particularly inspired by results on random walks on graphs, where the distinction between local and global statistical properties is well developed. For a simple random walk on a finite graph, global properties such as mixing times and spectral gaps are controlled by the spectrum of the transition matrix, while local quantities such as hitting, cover, and blanket times are much more sensitive to geometric features: bottlenecks, long sparse paths, and other ``rare'' structures. Proving sharp estimates for these local times typically requires combinatorial or geometric insights rather than purely spectral ones.

The theory of global statistical properties in ergodic theory grew directly out of the probabilistic viewpoint (via transfer operators and Markov approximations). One of our motivations here is to push this analogy further on the \emph{local} side: to understand which aspects of cover and blanket times for random walks have meaningful analogues for deterministic dynamical systems, and how the techniques must change.

A central motivation comes from the theory of simple random walks on finite graphs and the celebrated theorem of Ding, Lee and Peres \cite{DLP}. In the context of simple random walks on graphs, the cover and $(1-\epsilon)$-approximate blanket times can be defined analogously \S\ref{heuristics}. In \cite{WinZuc96}, Winkler and Zuckerman made the following conjecture, which was later proved in \cite{DLP}.

\begin{conj}\label{wz}
For each $\epsilon\in (0,1)$ there is a constant $K_\epsilon<\infty$ such that for every finite graph $G$ with set of vertices $V$
\[
\max_{v \in V}\mathbb{E}_v((1-\epsilon)\text{-approximate blanket time of $G$})
  \;\le\; K_\epsilon \cdot 
  \max_{v \in V}\mathbb{E}_v(\text{cover time of $G$}).
\]
\end{conj}

Informally: once the random walk has visited every vertex of the graph, one only has to wait a further factor $K_\epsilon$ (depending on $\epsilon$ but \emph{not} on the graph) to see an approximately stationary number of visits to each vertex.

An analogous question in our dynamical setting is:

\begin{ques}\label{metaques}
For which dynamical systems $(f,X,\mu)$ does there exist, for each $\epsilon\in (0,1)$, a constant $K_\epsilon<\infty$ (depending on $\epsilon$ but \emph{not} on the scale) such that for all sufficiently small scales $\delta \in (0,1)$,
\[
  \mathbb{E}_\mu \C_\delta 
  \;\leq\; \mathbb{E}_\mu \B_\delta^{1-\epsilon}
  \;\leq\; K_\epsilon\,\mathbb{E}_\mu \C_\delta,
\]
and how does $K_\epsilon$ depend on $\epsilon$?
\end{ques}

Whereas in the graph setting the significance of the Ding–Lee–Peres theorem lies in the fact that $K_\epsilon$ is uniform across all graphs, in our setting the property of interest is whether $K_\epsilon$ remains \emph{independent of the scale} $\delta$.

We emphasise an important conceptual difference between the graph and dynamical settings. In the graph setting, it is natural to fix a starting vertex and study extremal quantities such as $\max_{v \in V} \mathbb{E}_v(\cdot)$, leading to statements which are uniform over all initial conditions and all graphs. 

In contrast, in the dynamical setting there is typically no distinguished family of initial points, and orbit-level behaviour is most naturally studied from a measure-theoretic (i.e.\ typical) perspective. Accordingly, we consider expectations with respect to $\mu$, rather than extremal initial conditions. It is therefore not immediate how directly comparable the two formulations are, and part of the motivation for Question~\ref{metaques} is to understand to what extent an analogue of the graph-theoretic phenomenon persists in this setting.

\textbf{Notation:} For sequences $(a_n)_n$ and $(b_n)_n$ we write $a_n\lesssim b_n$ or $a_n\gtrsim b_n$ if there exists some $C>0$ such that $a_n\le Cb_n$ for all $n$, or $a_n\ge Cb_n$ for all $n$, respectively.  We write $a_n\asymp b_n$ if $a_n\lesssim b_n$ and $a_n\gtrsim b_n$: however if we want to emphasise constants we may also write $a_n= C^\pm b_n$ to mean $b_n/C\le a_n\le Cb_n$.  We sometimes write $a_n= O(b_n)$ if $|a_n|\lesssim |b_n|$.

\subsection{Our family of systems and results}\label{maps}

To investigate Question~\ref{metaques}, we consider the simplest family of systems that admits genuinely diverse multifractal behaviour. Since the multifractal structure of $\mu$ plays a key role in orbit-level stopping time problems, the full richness of the theory is not visible for Ahlfors regular measures (a class which, in the present setting, includes all acips).

\begin{assumptions} (For a more complete discussion see Section~\ref{sec:setup}.) We assume $I \subset \R$ is an interval\footnote{Although we do not consider maps whose repeller is a Cantor set, the arguments should extend to that setting with only minor modifications. Additional care would be required in constructing a suitable discretisation of the measure in \S\ref{mfsection}, due to the lack of interval structure.}  $f:I \to \R$ is map with finitely many, piecewise $C^2$ full branches with bounded distortion. We let $\phi:I \to \R$ be a H\"older potential function and we define the \emph{pressure} by 
$$P(\phi):=\lim_{n \to \infty} \frac{1}{n} \log \sum_{f^nx=x} e^{S_n\phi(x)}$$
where $S_n \phi$ denotes the ergodic sum $S_n\phi(x)=\sum_{i=0}^{n-1}\phi(f^i(x))$. We let $\mu=\mu_\phi$ be the unique invariant Gibbs measure for the potential $\phi$ (defined in the next section) and consider the dynamical system $(f,\mu)$.  We will also use the unique normalised $(\phi-P(\phi))$-conformal measure $m= m_\phi$.
\end{assumptions}

For the measure $\mu$, we write $d=d(\mu)$ and $D=D(\mu)$ to denote
\begin{equation}\label{minkowski}
 d:= \lim_{\delta \to 0} \frac{\log \max_{x \in I }\mu(B(x,\delta))}{\log \delta}\;\;\; \textnormal{and} \;\;\; D:= \lim_{\delta \to 0} \frac{\log \min_{x \in I }\mu(B(x,\delta))}{\log \delta}.
\end{equation}
It is known that in our setting the limits exist. Clearly $d \leq D$ and moreover these correspond to the left and right endpoints of the coarse multifractal spectrum, see \S \ref{mfsection} for further details. In particular, $D$ is sometimes referred to as the Minkowski dimension of the measure, see \cite{ffk}.

In \cite{JurTod25} the authors showed that in the current setting $\lim_{\delta \to 0} \frac{\log \mathbb{E}_\mu \C_\delta}{-\log \delta}= D$ and moreover\footnote{We note that \eqref{coverbounds} follows directly from \cite[Theorem 2.2]{JurTod25} combined with the fact that the Gibbs property implies that the error terms in \cite[Theorem 2.2]{JurTod25} disappear. This is because the quasi-Bernoulli property of Gibbs measures implies that the measure of the ball of
minimum measure at scale $\delta$ resembles the asymptotic limit $O(\delta^D)$ already at large scales $\delta$ (see for instance Lemma \ref{lem:weakreg}).} that
\begin{equation} \label{coverbounds}\delta^{-D} \lesssim\mathbb{E}_\mu \C_\delta \lesssim \delta^{-D} \log(1/\delta).\end{equation}

In the case where $\mu$ is an acip, the sharp asymptotic was proved that\footnote{This sharp asymptotic is a consequence of the fact that Alfhors regularity of an acip ensures that there are roughly $1/\delta$ balls of minimum measure at scale $\delta$ and this may not be the case more generally, see for example \cite[(1.2)]{chaosgame}.}
\begin{equation}
\mathbb{E}_\mu \C_\delta \asymp \delta^{-D} \log(1/\delta).
\label{covereqn}
\end{equation} 

Since covering is a prerequisite for blanketing, we always have 
$\mathbb{E}_\mu \C_\delta \leq \mathbb{E}_\mu \B_\delta^{1-\epsilon}$.
Thus, obtaining matching upper bounds on the blanket time is the central problem.

In the general multifractal setting, the lack of sharp asymptotics for the cover time prevents a direct comparison between cover and blanket times. Accordingly, our main result provides an upper bound on the blanket time at the correct scale. In the special case of acips, where sharp estimates for the cover time are available, this yields comparability between the two quantities.

\begin{thm} Suppose we have a system $(f,I,\mu)$ which satisfies the standing assumptions. Then
\[
\mathbb{E}_\mu \B_\delta^{1-\epsilon}
  \;\leq\; K_\epsilon\delta^{-D}\log \left(\frac1\delta\right)
\]
where $$K_\epsilon=O\left(\frac{1}{\epsilon^{2+\frac{1}{d}}}\right).$$
\label{thm:main}
\end{thm}

When $\mu$ is an acip then $d=D=1$ and the sharp estimate \eqref{covereqn} is known to hold. Therefore we obtain the following corollary to Theorem \ref{thm:main}, which gives an affirmative answer to Question~\ref{metaques} in this case.
\begin{cor} If $\mu$ is an acip then
\[
 \delta^{-1}\log(1/\delta)
  \;\lesssim\; \mathbb{E}_\mu \B_\delta^{1-\epsilon}
  \;\lesssim\; \frac{1}{\epsilon^3}\,\delta^{-1}\log(1/\delta) ,
\]
in particular
$$ \mathbb{E}_\mu \C_\delta\;\lesssim\; \mathbb{E}_\mu \B_\delta^{1-\epsilon}\lesssim \frac{1}{\epsilon^3} \mathbb{E}_\mu \C_\delta .$$
\end{cor}

\textit{Structure of the paper.} In \S \ref{outline} we will describe the general outline of our proof and state the two key technical Propositions \ref{prop:mftoolintro} and \ref{prop:ChaLep}, which will be required to prove Theorem \ref{thm:main}. In \S \ref{section:properties} we will describe some useful properties that our systems satisfy. In \S \ref{section:mainproof} we will give the proof of Theorem \ref{thm:main}. \S \ref{section:largedev} contains the proof of Proposition \ref{prop:ChaLep} and \S \ref{mfsection} contains the proof of Proposition \ref{prop:mftoolintro}.

\section{Proof outline and key propositions}\label{outline}

In this section we explain the previous proofs in the graph setting and give the two main propositions (Propositions~ and \ref{prop:ChaLep}) needed to adapt this approach in the dynamical case.  We also give Theorem~\ref{thm:easierproof} which is a simpler version of Theorem~\ref{thm:main} and closer in spirit to the graph results.

\subsection{Relation to the graph setting} \label{heuristics}

Let $G$ be a graph with vertex set $V$ and consider a simple random walk on $G$. Let $(\pi_v)_{v\in V}$ denote its stationary distribution. For $v \in V$, the hitting time $H_v$ is the first time that the vertex $v$ is visited by the walk.  The cover time $C_G$ is the first time that each vertex in $V$ is visited by the walk. The $(1-\epsilon)$-approximate blanket time is 
$$B^{1-\epsilon}:=\inf\{t>0: (\forall v) \; N_v(t) \geq t(1-\epsilon)\pi_v\}$$
where $N_v(t)$ denotes the number of times that the walk has visited $v$ up to time $t$. 

In \cite{WinZuc96}, Winkler and Zuckerman posed Conjecture \ref{wz} and proved it for simple random walks which satisfy
\begin{equation} \label{eq:wzineq}
\max_{v \in V} \E_v(C_G)\asymp\max_{v,w \in V} \E_v(H_w) \cdot \log(\# V).
\end{equation}
Since the backbone of our argument is based upon their approach, in order to orient the reader we will now sketch their proof for a simple random walk on a complete graph (such random walks are known to satisfy \eqref{eq:wzineq}). When we pass to the deterministic dynamical setting, this framework breaks down in two essential ways.

\begin{itemize}
\item In the graph case, the stationary distribution has finitely many weights $(\pi_v)_{v\in V}$. In our setting, at scale $\delta$ the role of $\pi_v$ is played by the continuum of local masses $\{\mu(B(x,\delta)) : x\in X\}$ at scale $\delta$, which typically exhibits multifractal behaviour. To estimate the blanket time we must control visit frequencies to all uncountably many balls of radius $\delta$. The core difficulty is therefore geometric rather than probabilistic, arising from the fine structure of the measure. Namely it will be necessary to identify a ``multifractal discretisation'' of our measure: a finite, geometry-dependent family of reference sets whose blanket guarantees global blanketing (blanketing of all $\delta$-balls), and moreover do so in a way which does not introduce scale-dependent error. The dependence of the blanket time on $\epsilon$ is governed by the multifractal spectrum of $\mu$, a phenomenon without a graph-theoretic analogue.
\item We will shortly see that in the graph case, the analysis boils down to studying the concentration properties of the hitting times to the states $v \in V$, and showing that these concentration properties are uniform in $v$.  In the dynamical setting, on the other hand, the main work is in proving large deviation-type estimates on hitting times to a collection of $\delta$-sized intervals, crucially with uniform bounds. The extra dependence in the dynamical setting here means these bounds require more work: note we use general Gibbs measures, which entail dependence on infinitely many time steps, rather than for example Bernoulli measures which only depend on the current time step.
 \end{itemize}

Let us now proceed with estimating $\max_{v \in V} \mathbb{E}_v( B^{1-\epsilon})$ for a simple random walk on a complete graph with $M$ vertices. We have $\#V = M$ and for each $v \in V$, $\pi_v=1/M$. Note that \eqref{eq:wzineq} says $\max_{v \in V} \E_v(C_G) \asymp M \log M$. 

We let  $H_v^n$ denote the $n$th hitting time to $v$. Let $w \in V$. Then
\begin{align*}
\E_w(B^{1-\epsilon})= \sum_{N=1}^\infty \mathbb{P}_w(B^{1-\epsilon} \geq N)& \leq \sum_{N=1}^\infty \mathbb{P}_w((\exists v): H_v^{\lceil N \pi_v(1-\epsilon)\rceil} \geq N) \nonumber\\
& \leq  \sum_{N=1}^\infty  \sum_{v \in V} \mathbb{P}_w( H_v^{\lceil N \pi_v(1-\epsilon)\rceil} \geq N).
\end{align*}
For $N>1$ we can count from when the walk has visited the vertex $v$ for the first time (so that the conditioning on the probabilities matches the vertex whose hitting we are estimating) and then by resumming over a sparser sequence of $N$ we have that the final expression is up to a uniform multiplicative constant less than
\begin{equation} \sum_{N=1}^\infty \sum_{v \in V}\mathbb{P}_v\left( H_v^{N } \geq \frac{N}{\pi_v} (1+\epsilon)\right) \frac{1}{\pi_v}. \label{keymarkov}
\end{equation}
Write $Y_{v}^N$ to be the random variable defined as the sum of $N$ independent copies of
$$Y_v:= H_v-(1+\epsilon)\frac{1}{\pi_v},$$
so that by Chernoff we have
\begin{equation} \label{ldpmarkov}
\mathbb{P}_v\left( H_v^{N } \geq \frac{N}{\pi_v} (1+\epsilon)\right)=\mathbb{P}_v(Y_{v}^N \geq 0) \leq \inf_{t>0} \E_v(e^{tY_{v}^N})=\left(\inf_{t>0} \E_v(e^{tY_v})\right)^N.
\end{equation}
Expanding $\E_v(e^{tY_v})$ as Taylor series and minimising the expansion over $t>0$, one can deduce that there exists $\alpha>0$, which is independent of $M$, such that
$$\mathbb{P}_v\left( H_v^{N } \geq \frac{N}{\pi_v} (1+\epsilon)\right) \leq e^{-\epsilon^2 N \alpha}.$$
Finally putting this into \eqref{keymarkov} gives
$$\E_w(B^{1-\epsilon}) \lesssim \sum_{N=1}^\infty \sum_{v \in V} e^{-\epsilon^2 N \alpha} \frac{1}{\pi_v} \approx \frac{M \log M}{\epsilon^2 \alpha}  \sum_{N=1}^\infty \sum_{v \in V} e^{-N \log M}=  \frac{M \log M}{\epsilon^2 \alpha} \sum_{N=1}^\infty e^{(1-N) \log M},$$
giving the result since $\sum_{N=1}^\infty e^{(1-N) \log M}$ is uniformly bounded over $M \in \N$.

This proves that for a simple random walk on a complete graph,
\begin{equation}\label{eq:markovbound}
\max_{v \in V}\mathbb{E}_v B^{1-\epsilon}\leq K_\epsilon \max_{v \in V} \mathbb{E}_v C_G \;\; \textnormal{where} \; \; K_\epsilon=O(1/\epsilon^2).
\end{equation}
Kolchin, Sevast'yanov and Chistyakov \cite{kolchin} showed that the dependence $K_\epsilon=O(1/\epsilon^2)$ is tight for complete graphs. Ding, Lee and Peres later proved that the same bound holds uniformly for simple random walks on arbitrary graphs, although its optimality in that general setting remains open.

\subsection{The dynamical setting}

In the Markovian blanket time problem, finiteness of the state space plays a decisive simplifying role: one only needs to track frequencies of visits to finitely many sets. We now set up a simplified version of our main result Theorem \ref{thm:main}, which more closely resembles this Markovian result, in that at any scale, the frequencies of visits to only finitely many balls are tracked. In this setting we recover the $1/\epsilon^2$ factor that appeared in the Markov case \eqref{eq:markovbound}.

  For each $\delta \in (0,1)$ let $\A_\delta$ be any minimal $\delta$-cover of $I$ i.e. a family of balls of radius $\delta/2$, each of which is centred in $I$, such that their union covers $I$ and this cannot be achieved by using fewer balls. 

Define the $\epsilon$-approximate $\mathcal{A}_\delta$-restricted blanket time $ \B_{\A_\delta}^{1-\epsilon}$ by
\[
  \B_{\A_\delta}^{1-\epsilon}(x):=\min\left\{N \ge 1: 
      \forall A \in \A_\delta,\ 
      \frac{1}{N} \sum_{i=0}^{N-1} 
        \mathbf{1}_{A}(f^i x) 
      \geq (1-\epsilon)\,\mu\bigl(A\bigr)
  \right\}.
\]
As an example, one may consider the level-$n$ dyadic partition $\A_{1/2^n}$ in which case $\B_{\A_{1/2^n}}^{1-\epsilon}$ records the time that every interval in the level-$n$ dyadic partition has been visited with roughly the correct frequency.  Then the following result is analogous to \eqref{eq:markovbound}.

\begin{thm}\label{thm:easierproof}
Suppose $(f,I,\mu)$ satisfies the same conditions as in the statement of Theorem \ref{thm:main}. Then 
$$\mathbb{E}_\mu \B_{\A_\delta}^{1-\epsilon} \lesssim \frac{1}{\epsilon^2} \frac1{\delta^D}\log\left(\frac1\delta\right).$$
\end{thm}

In the remainder of this section we set up the propositions needed to prove our main theorems in the dynamical case.

\subsubsection{Constructing a multifractal discretisation of $\mu$}

The first key challenge in studying the blanket time in the dynamical setting is that by definition, we must keep track of frequencies of visits to a continuum of balls $\{\mu(B(x, \delta))\}$. To circumvent this, for any $\epsilon, \delta \in (0,1)$ we will construct a finite reference set $\mathcal{A}_{\delta, \epsilon}$ whose blanketing guarantees global blanketing. This is done by constructing a finite family of sets with the property that for any arbitrary $\delta$-ball $B$, a set in this collection can be found which is contained in and nearly exhausts the measure of $B$.

\begin{defn} \label{mfd}
An \emph{$(\epsilon, \delta)$-multifractal discretisation of $\mu$} is a collection $\A_{\delta,\epsilon}$ of sets, each of diameter at least $\delta/2$, such that for any $\delta$-ball $B(x,\delta)$, there exists $A \in \A_{\delta, \epsilon}$ such that $A \subset B(x,\delta)$ and
$$\frac{\mu(A)}{\mu(B(x,\delta))} \geq 1-\epsilon.$$
\end{defn}

Note that this discretisation generally will not be a partition (unless, for example, $\epsilon$ is very large and $\mu$ is very regular).

This definition is useful because one can verify that if we $(1-\epsilon)$-blanket an $(\epsilon, \delta)$-multifractal discretisation, then we $(1-2\epsilon)$-blanket globally at scale $\delta$. 

Now, although an $(\epsilon,\delta)$-discretisation is not difficult to construct
(for example, by choosing the start points of the sets in
$\mathcal{A}_{\delta,\epsilon}$ sufficiently dense so that
Definition~\ref{mfd} is satisfied), the cardinality
$\#\mathcal{A}_{\delta,\epsilon}$ directly affects the resulting bound on the
blanket time and therefore must be kept under control.

Indeed, the quantity
\[
\sum_{N=1}^{\infty} \sum_{v \in V} e^{-N \log M}
  = \sum_{v \in V} e^{-\log M}
    \sum_{N=1}^{\infty} e^{-(N-1)\log M}
\]
which appears in the graph setting will be replaced in the dynamical
setting by
\begin{equation}
\sum_{A \in \mathcal{A}_{\delta,\epsilon}}
  e^{-\mathbf{C}_\delta \mu(A)}
  \sum_{N=1}^{\infty} e^{-(N-1)\mathbf{C}_\delta \mu(A)}.
\label{doublesum}
\end{equation}
where $\mathbf{C}_\delta$ will be some constant of the order $\mathbf{C}_\delta  \asymp \delta^{-D}\log(1/\delta)$. 
A crude estimate now gives
\[
\sum_{A \in \mathcal{A}_{\delta,\epsilon}}
  e^{-\mathbf{C}_\delta \mu(A)}
  \le
  \#\mathcal{A}_{\delta,\epsilon}\,\delta.
\]
Thus, proving Theorem~\ref{thm:main} via this bound would require
$\#\mathcal{A}_{\delta,\epsilon} = O_\epsilon(1/\delta)$.

This growth rate for $\# \A_{\delta,\epsilon}$ indeed occurs when $\mu$ is an acip.
However, for more general multifractal measures, the cardinality of a multifractal
discretisation must necessarily grow faster with $\delta$.
Consequently, a more delicate estimate of the sum
\eqref{doublesum} is required.

\begin{prop}\label{prop:mftoolintro}
Letting  $\mathbf{C}_\delta= p\delta^{-D}\log(1/\delta)$ for an appropriate choice of $p$, the following holds.  For all $\epsilon,\delta \in (0,1)$ one can construct an
$(\epsilon,\delta)$-multifractal discretisation of $\mu$ such that
\[
\sum_{A \in \mathcal{A}_{\delta,\epsilon}}
  e^{-\mathbf{C}_\delta \mu(A)}
  = O(\epsilon^{-1/d}).
\]
\end{prop}

The proof of this proposition requires non-trivial arguments from
multifractal analysis, see \S \ref{mfsection}.

\subsubsection{Uniform large deviation estimates for hitting times}
A large proportion of the proof of our main theorems will be devoted to obtaining an analogue of \eqref{ldpmarkov} in the dynamical setting. Namely, given a subinterval $A \subset I$ and letting $\H_A^n$ be the $n$th hitting time to $A$, we will need to obtain estimates on
$$ \mu\left( \H_A^{n } \geq n\left(\frac{1}{\mu(A)}+\eta\right)\right)$$
which is uniform over all $A$ of comparable diameters. Since $\H_A^n$ can be written as an ergodic sum for the first return dynamics, and $\mathbb{E} \H_A^n=n/\mu(A)$, this can be interpreted as a large deviation estimate. Thus our main tool is a large deviations result which is  related to \cite[Theorem 2.2]{ChaLep05} in the setting of Axiom A diffeomorphisms, but here we give an independent, and more direct, proof in order to keep track of the constants. We define Markov intervals in the next section.

\begin{prop}
There exist $c_1,c_2, c_3>0$ such that for each sufficiently small Markov interval $A\subset X$,  for any $\eta<\frac{c_1}{\mu(A)}$,
$$ \mu\left( \H_A^{n } \geq n\left(\frac{1}{\mu(A)}+\eta\right)\right)\le c_3e^{-c_2\eta^2\mu(A)^2n}.$$
\label{prop:ChaLep}
\end{prop}

\section{Properties of our systems and measures}\label{section:properties}

\label{sec:setup}

We suppose throughout that $I$ can be decomposed into a finite partition $\P$ where if $P\in \P$, then $f:P\to I$ is  $C^2$, uniformly expanding, and $f$ extends continuously to $\overline{P}$ such that $f:\overline{P}\to I$ is a diffeomorphism.

For $n \in\N$ let $\P_n:= \bigvee_{i=0}^{n-1}f^{-i}\P$ and refer to $P\in \P_n$ as an \emph{$n$-cylinder}.  We call an interval $A\subset I$ \emph{Markov} if it can be expressed as the union of $k$-cylinders, for some $k\in \N$. We assume we have  bounded distortion in the sense that
$$\sup_{P\in\P_n}\sup_{x, y\in P}\frac{|Df^n(x)|}{|Df^n(y)|}<\infty.$$

\subsection{Markov shifts and coding}

Given an alphabet $\mathbb{A}$, which will be finite or countable and a $\#\mathbb{A}\times \#\mathbb{A}$ matrix $(a_{ij})_{i, j}$ of 0s and 1s, the corresponding shift space is
$$\Sigma:= \left\{(i_1, i_2, \ldots): i_k\in \mathbb{A} \text{ and } a_{i_k, i_{k+1}} = 1 \text{ for all } k\ge 1\right\},$$
and $\sigma:\Sigma\to \Sigma$ is the left shift.  An \emph{allowed word $\i= i_1\cdots i_k$ of length $k$} is such that $i_j\in \mathbb{A}$ for $1\le j\le k$ and  $a_{i_j, i_{j+1}} = 1$ for $1\le j< k$. Let $\Sigma_k$ be the set of all such words and define $\Sigma^*:= \cup_{k\ge 1} \Sigma_k$.  For $\i\in \Sigma^*$, we write $|\i |= k$ if $\i\in \Sigma_k$.  If all words are allowed, i.e., $\Sigma = \mathbb{A}^{\N}$,  then we call $(\Sigma, \sigma)$, or just $\Sigma$, the  \emph{full shift}.

Given our system $(I, f)$, we can code this by writing $\P= \{P_1, \ldots, P_{\#\P}\}$ and for $\Sigma= \{ 1, \ldots, \#\P\}^{\N}$, the full shift, for $(i_1, i_2, \ldots)\in \Sigma$ we define $\Pi(i_1, i_2, \ldots) = x\in I$ where $f^{k-1}(x)\in P_{i_k}$ for all $k$.  Similarly define $\Pi(\i)$ for $\i\in \Sigma^*$ and note that this corresponds to some element of $\P_k$ if $|\i|=k$.  We note that there may be issues with coding at the boundaries of these intervals, but they will be of zero measure and we can make arbitrary choices there.

Later we will also code induced dynamics by countable Markov shifts. To this end, we say that the shift is \emph{topologically transitive} if for all $a, b\in \mathbb{A}$ there is some $\i\in \Sigma_k$ for some $k$  such that $\i= ai_2\ldots i_{k-1}b$.  In case $\mathbb{A}$ is countable, it is useful to define 
 Sarig's BIP property, see \cite{Sar03}: our shift satisfies the \emph{BIP property} if there exist a finite set $b_1, \ldots, b_N\in \mathbb{A}$ such that for any $a\in \mathbb{A}$, there are $1\le i, j\le N$ such that $b_iab_j\in \Sigma_3$.

\subsection{Thermodynamic formalism and measures}

 Given $\psi:I\to \R$, let  $$V_n(\psi):= \sup_{P\in \P_n}\sup_{x, y\in P}|\psi(x)-\psi(y)|$$ be the \emph{$n$-th variation} of $\psi$ and say that $\psi$ \emph{is locally H\"older continuous} if there is $\ell>0$ such that $ V_n(\psi)= O(e^{-\ell n})$ for $n\ge 1$.  We say that a measure $m_\psi$ on $I$ is \emph{$\psi$-conformal} if
$$m_\psi(f(A))=\int_A e^{-\psi} dm_\psi$$
for all Borel measurable $A \subset I$ such that $f:A \to f(A)$ is a bijection.

We suppose that $\phi:I\to \R$ is H\"older on each $P\in \P$.  Possibly after replacing $\phi$ with $\phi-P(\phi)$, we assume that $P(\phi)=0$.  Uniform expansion, bounded distortion and the H\"older property here imply that $\psi$ is also locally H\"older and there is a $\phi$-conformal measure $m=m_\phi$, see for example \cite{Ryc83}, and a corresponding invariant probability measure $\mu_\phi \ll m_\phi$.  In line with that paper, we define the operator $\L= \L_\phi: \mathfrak{B} \to \mathfrak{B}$
$$\L g(x):=\sum_{y \in f^{-1}x} e^{\phi(y)} g(y)$$
on the Banach space of functions with finite norm $\|\cdot\|=\|\cdot\|_{1}+TV(\cdot)$ where $\|\cdot\|_1$ is the $L^1$ norm with respect to $m$ and $TV$ is the total variation.  Then there exists a strictly positive and bounded function $\rho \in \mathfrak{B}$ such that $\L \rho=\rho$, $m_\phi$ is  the fixed point of the dual $\L_\phi^*m_\phi=m_\phi$, and $d\mu_\phi=\rho dm_\phi$. Moreover $\mu=\mu_\phi$ is \emph{Gibbs} for the potential $\phi$, i.e. there exists $0<C<\infty$ such that for each $P \in \P_n$ and $x\in P$,
$$\frac{1}{C} \leq \frac{\mu(P)}{e^{S_n \phi(x)}} \leq C.$$
We will also use the  \emph{quasi-Bernoulli} property that follows from this: there exists $c>0$ such that for $\i, \j\in \Sigma^*$
$$\frac{1}{c}\mu(\Pi(\i)) \mu(\Pi(\j)) \leq\mu(\Pi(\i \j)) \le c\mu(\Pi(\i)) \mu(\Pi(\j)).$$

We note that most of what we have presented so far can be obtained from the theory of Mauldin-Urba\'nski and Sarig (see for example \cite{MauUrb01} and \cite{Sar99}) where locally H\"older continuous (or summable variation) is the appropriate condition on the observables used, but as in many studies of hitting times or open systems, Rychlik's theory is preferable since the perturbation to the system of introducing a hole/target is less drastic in the associated Banach space (this will be useful in Proposition~\ref{prop:PP} below).  Note that this is true even if the hole, which we will assume is an interval, is Markov.

We will require the following basic regularity property of the measure $\mu$. We let $|\cdot|$ denote the diameter of a set.

\begin{lma} \label{lem:weakreg}
There exist $0<a<b<\infty$ such that for any interval $U$ with diameter $|U|$ then 
\begin{equation*}
a|U|^{D} \leq \mu(U) \leq b |U|^{d}.
\label{eq:weakreg}
\end{equation*}
\end{lma}
\begin{proof}
We begin with the left hand side. Let $U=B(x,r)$ for some $x \in I$. Note that by definition of the Minkowski dimension, for all $\epsilon>0$ there exists $k=k(\epsilon)$ such that for all $\i \in \Sigma^*$, $|\i|\ge k$, 
$$\mu(\Pi[i_1 \ldots i_k]) \geq |\Pi[i_1 \ldots i_k]|^{D+\epsilon}.$$
Fix $x \in X$, $r>0$ and $\Pi(\i)=x$. We can find $n \in \N$ such that $\Pi[i_1 \ldots i_n] \subset B(x,r)$ and $|\Pi[i_1 \ldots i_n]| \geq \kappa r$ (where the constant $\kappa \in (0,1)$ is independent of $x$ and $r$ because our map has finitely many uniformly expanding branches).

Let $\epsilon>0$ be arbitrary and fix $k(\epsilon)=k$ as above. By the quasi-Bernoulli property,
\begin{align*}
\mu(\Pi[i_1 \ldots i_n])^k \geq c^k\mu(\Pi[(i_1 \ldots i_n)^k])\geq c^k|\Pi[(i_1 \ldots i_n)^k]|^{D+\epsilon} \geq c^k(\kappa^kr^k)^{D+\epsilon}
\end{align*}
where the constant in the final line comes from considering the bounded distortion property for $(f^n)'$ and applying the chain rule to $(f^{nk})'$. Therefore
$$\mu(\Pi[i_1 \ldots i_n]) \geq c\kappa^{D+\epsilon}r^{D+\epsilon}.$$
But since $c$ and $\kappa$ were independent of $\epsilon$ and $\epsilon>0$ was arbitrary,
$$\mu(B(x,r)) \geq\mu(\Pi[i_1 \ldots i_n]) \geq c\kappa^{D}r^{D}.$$
The right hand side is proved similarly by instead considering upper bounds and substituting $d$ for the Minkowski dimension $D$.
\end{proof}

\section{Proof of the main results}\label{section:mainproof}

We begin with the proof of our simplified result Theorem~\ref{thm:easierproof}, which estimates the expected $\mathcal{A}_\delta$-restricted blanket time, where we quantify how quickly the visits to a specific representative cover or partition $\mathcal{A}_\delta$ are of roughly the correct frequency.

\begin{proof}[Proof of Theorem~\ref{thm:easierproof}]
Fix $\epsilon, \delta \in (0,1)$. It is easy to see that given any interval in $A \in\A_\delta$ we can find a set $U\subset A$ such that $|U|>\delta/2$, $\frac{\mu(U)}{\mu(A)} \geq 1-{\epsilon}$ and $U$ can be realised as a finite union of cylinders. Let $\U_\delta$ denote the corresponding collection of such sets $U$.  

Note that if $N>\mathcal{B}_{\U_\delta}^{1-\epsilon}(x)$ then for any $A \in \A_\delta$, there exists $U \in \U_\delta$ such that $U \subset A$ and
$$\frac{1}{N} \sum_{i=0}^{N-1} \mathbf{1}_A(f^ix) \geq \frac{1}{N} \sum_{i=0}^{N-1} \mathbf{1}_U(f^ix) \geq \left(1-{\epsilon}\right)\mu(U) \geq \left(1-{\epsilon}\right)^2\mu(A)\geq \left(1-2{\epsilon}\right)\mu(A)$$
and hence 
$$\mathbb{E}_\mu\left(\B_{\A_\delta}^{1-2{\epsilon}}\right) \leq \mathbb{E}_\mu\left(\B_{\U_\delta}^{1-\epsilon}\right).$$
Thus we can estimate
\begin{align*}
\mathbb{E}_\mu(\B_{\A_\delta}^{1-2{\epsilon}}) \leq \mathbb{E}_\mu(\B_{\U_\delta}^{1-\epsilon})   &= \sum_{N=1}^\infty \mu\left(\B_{\U_\delta}^{1-\epsilon} \geq N\right)\leq \sum_{N=1}^\infty \sum_{U \in \U_\delta} \mu\left(\H_U^{\lceil N \mu(U)(1-\epsilon)\rceil} \geq N\right).
\end{align*}

 By Proposition~\ref{prop:ChaLep}, there exist $c_1,c_2,c_3>0$ such that for $\epsilon<c_1$ and a Markov interval $A$,
$$
\mu\left( \H_A^{M } \geq \frac{M}{\mu(A)} (1+\epsilon)\right) \le c_3 e^{-\epsilon^2Mc_2}
$$
for any $M \in \N$. Therefore, for each $U \in \U_\delta$ we can set $M=\lceil N \mu(U)(1-\epsilon)\rceil$ to obtain
\begin{align*}
\mu\left( \H_A^{\lceil N \mu(U)(1-\epsilon)\rceil} \geq N\right) &\leq \mu\left( \H_A^{\lceil N \mu(U)(1-\epsilon)\rceil} \geq \frac{\lceil N \mu(U)(1-\epsilon)\rceil}{\mu(U)}(1+\epsilon)\right)  \nonumber\\
&    \le c_3 e^{-\epsilon^2\lceil N \mu(U)(1-\epsilon)\rceil c_2} \leq c_3e^{-2c_2\epsilon^2N \mu(U)}. 
\end{align*}
Let $a$ be given by Lemma \ref{lem:weakreg} and write  $\mathbf{C}_\delta=2^{D}a^{-1} \delta^{-D} \log (1/\delta)$. Then

\begin{align*}
\mathbb{E}_\mu(\B_{\A_\delta}^{1-2{\epsilon}}) &\leq \sum_{N=1}^\infty \sum_{U \in \U_\delta} c_3e^{-2c_2\epsilon^2N \mu(U)}\leq  \sum_{N=0}^\infty \sum_{U \in \U_\delta} c_3 e^{-N \mathbf{C}_\delta \mu(U)} \frac{1}{2c_2\epsilon^2} \mathbf{C}_\delta\\
&\lesssim \frac{ \mathbf{C}_\delta}{\epsilon^2}+ \frac{ \mathbf{C}_\delta}{\epsilon^2} \left( \sum_{N=1}^\infty \sum_{U \in \U_\delta}e^{-N \mathbf{C}_\delta \mu(U)}\right).
\end{align*}
Now note that
$$ \sum_{N=1}^\infty \sum_{U \in \U_\delta}e^{-N \mathbf{C}_\delta \mu(U)}= \sum_ {U \in \U_\delta} e^{-\mathbf{C}_\delta \mu(U)} \sum_{N=1}^{\infty} e^{-(N-1)\mathbf{C}_\delta \mu(U)}.$$
Since $|U| \geq \frac{\delta}{2}$, by Lemma \ref{lem:weakreg} 
$$\mathbf{C}_\delta\mu(U) \geq\mathbf{C}_\delta a \left|\frac{\delta}{2}\right|^{D}=\log(1/\delta)>1,$$
the second sum $\sum_N e^{-(N-1)\mathbf{C}_\delta\mu(U)}$ is uniformly bounded over all $U \in \U_\delta$ and $\delta \in(0,1)$. Since $\A_\delta$ was a minimal cover of $I$, $\#\U_\delta =\#\A_\delta\approx 1/\delta$, the first sum $\sum_ {U \in \U_\delta} e^{-\mathbf{C}_\delta\mu(U)}$ is also uniformly bounded over all $U \in \U_\delta$ and $\delta \in(0,1)$. This completes the proof.
\end{proof}

Let us now turn back to our original blanket time problem, in which we care about $\epsilon$-approximately equidistributing over \emph{all} balls. Recall that an $(\epsilon,\delta)$-multifractal discretisation of our measure (Definition \ref{mfd}) is a collection $\A_{\delta,\epsilon}$ of subintervals of $I$ such that for any $\delta$-ball $B(x,\delta)$, there exists $A \in \A_{\delta, \epsilon}$ such that $A \subset B(x,\delta)$ and $\frac{\mu(A)}{\mu(B(x,\delta))} \geq 1-\epsilon.$ Recall that $\A_{\delta,\epsilon}$ will generally not be a partition since its elements may be highly overlapping.

It is easy to see that 
$$ \mathbb{E}_\mu(\B_\delta^{1-2\epsilon}) \leq \mathbb{E}_\mu(\B_{\A_{\delta,\epsilon}}^{1-\epsilon})$$
since if $N>\mathcal{B}_{\A_{\delta,\epsilon}}^{1-\epsilon}(x)$ then for any $B(y,\delta)$ we can, by definition of a multifractal discretisation, find a set $A \subset B(y,\delta)$ such that
$$\frac{1}{N} \sum_{i=0}^{N-1} \mathbf{1}_{B(y,\delta)}(f^ix) \geq \frac{1}{N} \sum_{i=0}^{N-1} \mathbf{1}_A(f^ix) \geq (1-\epsilon)\mu(A) \geq (1-\epsilon)^2\mu(B(y,\delta))\geq (1-2\epsilon)\mu(B(y,\delta)).$$
Hence we can reduce the global blanket time problem to a restricted blanket time problem. By following the proof of Theorem \ref{thm:easierproof}, we should deduce an estimate which depends on how $\sum_ {A\in \A_{\delta,\epsilon}} e^{-\mathbf{C}_\delta\mu(A)}$ scales as $\delta,\epsilon \to 0$, for an appropriate $\mathbf{C}_\delta$. Combining this with Proposition \ref{prop:mftoolintro}, which ensures that a multifractal discretisation $\A_{\delta,\epsilon}$ can be found such that $\sum_ {A\in \A_{\delta,\epsilon}} e^{-\mathbf{C}_\delta\mu(A)}$ can be bounded uniformly over all small scales $\delta$, will complete the proof of Theorem \ref{thm:main}.

\begin{proof}[Proof of Theorem \ref{thm:main}]   
Fix a measure $\mu$ and let $\A_{\delta,\epsilon}$ be a multifractal decomposition of $\mu$ given by Proposition \ref{prop:mftoolintro}. Then
\begin{align*}
\mathbb{E}_\mu(\B_{\delta}^{1-2\epsilon}) \leq\mathbb{E}_\mu(\B_{\A_{\delta,\epsilon}}^{1-{\epsilon}})  &= \sum_{N=1}^\infty \mu\left(\B_{\A_{\delta,\epsilon}}^{1-\epsilon} \geq N\right)\leq \sum_{N=1}^\infty \sum_{A \in \A_{\delta,\epsilon}} \mu\left(\H_A^{\lceil N \mu(A)(1-\epsilon)\rceil} \geq N\right).
\end{align*}

 By Proposition~\ref{prop:ChaLep}, there exist $c_1,c_2,c_3>0$ such that for $\epsilon<c_1$  
$$
\mu\left( \H_A^{M } \geq \frac{M}{\mu(A)} (1+\epsilon)\right) \le c_3 e^{-\epsilon^2Mc_2}
$$
for any $A \in \A_{\delta,\epsilon}$ and $M \in \N$. Therefore, for each $A \in \A_{\delta,\epsilon}$ we can set $M=\lceil N \mu(A)(1-\epsilon)\rceil$ to obtain
\begin{align*}
\mu\left( \H_A^{\lceil N \mu(A)(1-\epsilon)\rceil} \geq N\right) &\leq \mu\left( \H_A^{\lceil N \mu(A)(1-\epsilon)\rceil} \geq \frac{\lceil N \mu(A)(1-\epsilon)\rceil}{\mu(A)}(1+\epsilon)\right)  \\
&    \le c_3 e^{-\epsilon^2\lceil N \mu(A)(1-\epsilon)\rceil c_2} \leq c_3e^{-2c_2\epsilon^2N \mu(A)}.
\end{align*}
Let $a$ be given by Lemma \ref{lem:weakreg} and write  $\mathbf{C}_\delta=2^{D}a^{-1} \delta^{-D} \log (1/\delta)$. Then
\begin{align*}
\mathbb{E}_\mu(\B_{\A_{\delta,\epsilon}}^{1-{\epsilon}}) &\leq \sum_{N=1}^\infty \sum_{A \in \A_{\delta,\epsilon}} c_3e^{-2c_2\epsilon^2N \mu(A)}\leq  \sum_{N=0}^\infty \sum_{A \in \A_{\delta,\epsilon}} c_3 e^{-N \mathbf{C}_\delta \mu(A)} \frac{1}{2c_2\epsilon^2} \mathbf{C}_\delta\\
&\lesssim \frac{ \mathbf{C}_\delta}{\epsilon^2}+ \frac{ \mathbf{C}_\delta}{\epsilon^2} \left( \sum_{N=1}^\infty \sum_{A \in \A_{\delta,\epsilon}}e^{-N \mathbf{C}_\delta \mu(A)}\right).
\end{align*}
Now note that 
$$ \sum_{N=1}^\infty \sum_{A \in\A_{\delta,\epsilon}}e^{-N \mathbf{C}_\delta \mu(A)}= \sum_ {A \in \A_{\delta,\epsilon}} e^{-\mathbf{C}_\delta \mu(A)} \sum_{N=1}^{\infty} e^{-(N-1)\mathbf{C}_\delta \mu(A)}.$$
As before, the lower bound on the diameter of $A$ means that the second sum can be uniformly bounded using Lemma \ref{lem:weakreg}. The first sum can be bounded by $O(\epsilon^{-1/d})$ by Proposition \ref{prop:mftoolintro}, so the leading term is $\frac{ \mathbf{C}_\delta}{\epsilon^{2+\frac1d}}$, as required.
\end{proof}

\section{The pressure function and measures}\label{section:largedev}

In this section we will  prove Proposition~\ref{prop:ChaLep}.   
 We start by defining some pressure functions for our original system and our induced system.  We will find uniform information on the shape of these functions, find corresponding invariant and conformal measures and uniform estimates on their densities, as well as estimates on the measures of certain cylinders.  
 
We will choose intervals $A$ and define $p_A(q):= P\left(\phi+q\left(1-\frac{1_A}{\mu(A)}\right)\right)$.

\begin{prop}  For any interval $A$, the following hold.
\begin{enumerate}
\item[(a)] There exists a $\left(\phi+q\left(1-\frac{1_A}{\mu(A)}\right)-p_A(q)\right)$-conformal measure $m_{A, q}$ and an equilibrium state $\mu_{A, q}$ for $\phi+q\left(1-\frac{1_A}{\mu(A)}\right)$.
\item[(b)] The function $p_A$ is analytic.
\end{enumerate}
\label{prop:PP}
\end{prop}

\begin{proof}
These facts follow from standard techniques of \cite[Chapter 4]{ParPol90} along with the fact that $\L_{A, q}:= \L_{\phi+ q\left(1-\frac{1_A}{\mu(A)}\right)}$ has a spectral gap for any $q$, since our function $\phi+q\left(1-\frac{1_A}{\mu(A)}\right)\in \mathfrak{B}$, see \cite{Ryc83}.  
\end{proof}

\subsection{Inducing}
We next take first return maps to suitable intervals $A$ and give some basic properties of these induced systems.

For a Markov interval $A \subset I$, which we assume can be written as $A=\bigcup_{i=1}^p A_i$ where $A_i\in\P_k$,  let $F_A:A \to A$ be the first return map to $A$, i.e. $F_A(x)=f^{\H_A(x)}(x)$.  Then there exists a natural maximal partition $\P^A$ so that $\H_A$ is constant on each element $P\in \P^A$ and $F_A\overline P = \overline{A_{i}}$ for some $i=1, \ldots, p$. The latter Markov property is inherited from the Markov nature of $A$.
We can label the elements of $\P^A= \P_1^A$ by $\{P_1, P_2, \ldots\}$ thus giving a coding by $i\in \N$.  Then the dynamics induces a coding of points in $A$ by $(i_1, i_2, \ldots)\in \N^{\N}$ where $\Pi(i_1, i_2, \ldots) =x$ if $F_A^k(x)\in P_{i_k}$ for all $k \in \N$.  We write the corresponding shift space as $(\Sigma^A, \sigma)$.

Let $\P_n^A$ denote the induced $n$-cylinders where we enumerate the depths and note any $P\in \P_n^A$ can be written as $\Pi[i_1, \ldots, i_n] = P$ for some $i_i\ldots i_n\in \Sigma_n^A$.  Hence $P\in \P_n^A$ implies there is $A_i$ with $F_A^n(\overline{P})=\overline{A_i}$.   We will sometimes emphasise the symbolic address here, referring to $\Pi[i_1, \ldots, i_n]$ and sometimes just write $P\in \P_n^A$.  We abuse notation by saying that $P_1P_2\ldots P_n$ is an allowable word if there is $x\in A$ with $F_A^{k-1}x\in P_k\in \P_1^A$ for $k=1, \ldots, n$.  From here on we will stop being concerned about what happens at boundary points,  further abusing notation to write, for example $F_A^nP=A_i$ for the above statement.  

Note that $\H_A^n=S_n \H_A$.
We define the induced potential $\phi_A$ by $\phi_A(x)= S_{\H_A(x)}\phi(x)$. 
It is useful here to point out that we use $S_n\psi= S_n^g\psi$ to denote the ergodic sum for the relevant dynamics $g$, so for example $\H_A^n=S_n^{F_A} \H_A$, while $\phi_A(x)= S_{\H_A(x)}^f\phi(x)$.
 We can develop the thermodynamic formalism for this system using the theory of Mauldin-Urba\'nski and Sarig or Rychlik \cite{Ryc83}, as mentioned above. 

We collect some standard useful properties of our induced system in the next proposition.

\begin{prop}
For a Markov interval $A =\bigcup_{i=1}^p A_i$ where $A_i\in\P_k$ and  $F_A:A \to A$ is the first return map to $A$,
\begin{enumerate}
\item[(a)] $(A, F_A)$ is coded by a shift with the BIP property;
\item[(b)] the induced potential $\phi_A$ is locally H\"older continuous;
\item[(c)] $P(\phi_A)=0$;
\item[(d)]  $\mu_A=\mu_{\phi_A}= \frac{1}{\mu(A)}\mu |_A$ is $F_A$-invariant and has the Gibbs property for $\phi_A$;
\item[(e)] $\E_{\mu_A}(\H_A)=\frac{1}{\mu(A)}$;
\item[(f)] $d\mu_A=\rho_A dm_A$ where $\rho_A=\frac{m(A)}{\mu(A)} \rho$ and $m_A=\frac{1}{m(A)} m|_A$. 
\item[(g)] $m_A$ is $\phi_A$-conformal;
\end{enumerate}
\label{prop:indprops}
\end{prop}

We note that (d) also implies the quasi-Bernoulli property.

\begin{proof}
Recall that from the Markov property of our intervals, for each $P\in\P_1^A$, $F_AP= A_i$ for some $A_i$. 
By the topological transitivity of our original system, for each $A_i$ there exists $P_i\in \P_1^A$ such that $F_AP_i = A_i$.  Then also choosing arbitrary $Q_i\in \P_1^A\cap A_i$, the collection $\{P_1, \ldots, P_k, Q_1, \ldots, Q_k\}$ can take the place of the symbols in the BIP property, since if $P\in \P_1^A$ is in $A_i$ and $F_AP=A_j$, then $P_iPQ_j$ is an allowable word, so the BIP property holds, i.e., (a) follows.

Part (b) is immediate since $V_n(\phi_A)\le V_n(\phi)= O(e^{-\ell n})$.

Part (c) follows for example by \cite[Lemma 4.1]{IomTod10}.  Parts (d), (e) and (f) follow from the Kac Lemma and the fact that $\mu_{\phi_A}$, the equilibrium state for $\phi_A$, must be Gibbs as in \cite{Sar03}.    Part (g) follows by the conformality of $m_\phi$ which implies that that for any $B \subset A$ such that $F:B \to F(B)$ is a bijection, 
$$m_A(F(B))=\frac1{m_\phi(A)}m_\phi(F(B))= \frac1{m_\phi(A)} \int_B e^{-\phi_A}dm_\phi=  \int_B e^{-\phi_A}dm_A$$
i.e., $m_A$ is $\phi_A$-conformal for the first return dynamics $F_A$. 
\end{proof}

We will later consider more general conformal measures defined as for the uninduced system, in particular for potentials $\Phi=\phi_A+q\H_A$, but before we can introduce these properly, we require more information on the pressure function in the next subsection.  However, it is useful to record here that 
if $C\in \P_n^A$ then $F_A^{n}C = A_i$ for some $1\le i\le p$, so if $m_\Phi$ is $(\Phi-P(\Phi))$-conformal then
$$m_\Phi(A_i)=m_\Phi(F_A^nC)  = \int_Ce^{-S_n\Phi+nP(\Phi)}~dm_\Phi = e^{\pm \sum_{n=1}^\infty V_n(\Phi)}m_\Phi(C)e^{-S_n\Phi(x)+nP(\Phi)} \text{ for } x\in C,$$
 so 
 \begin{equation} m_\Phi(C)= e^{\pm \sum_{n=1}^\infty V_n(\Phi)}m_\Phi(A_i)e^{S_n\Phi(x)-nP(\Phi)} \text{ for } x\in C.
 \label{eq:conf}
 \end{equation}

\subsection{The shape of the pressure function}
We will next use Proposition~\ref{prop:PP} to give properties of pressure functions like $q\mapsto P(\phi_A+q\H_A)$.

\begin{prop}
There exists $K>0$ such that for any $A$, the function $q\mapsto P(\phi_A+q\H_A)$ is analytic (and finite) for $q\in \mathbb{C}$ if $|q|\le K\mu(A)$.
Moreover, if $|q|\le K\mu(A)$ then $\left|P(\phi_A+q\H_A)\right| \le 2|q|/\mu(A)$.
\label{prop:ana}
 \end{prop}
 
\begin{proof}
By Proposition~\ref{prop:indprops}, our induced system is encoded by a BIP system, so since that proposition also says that $\phi_A$ is locally H\"older continuous,  the proof of \cite[Corollary 4]{Sar03} implies that for analyticity, we only need to show that there is $K>0$ such that $P(\phi_A+q\H_A)<\infty$ for all $|q|\le K\mu(A)$.  Moreover, monotonicity will imply that it is sufficient to prove that this is finite at $K\mu(A)$.

We first assume that $q\ge 0$.
Since, trivially, $P\left(\phi+q\left(1-\frac{1_A}{\mu(A)}\right) - p_A(q) \right)=0$, as in for example \cite[Lemma 4.1]{IomTod10}, inducing gives 
\begin{equation}
P\left(\phi_A+\H_A\left(q - p_A(q)\right)- \frac{q}{\mu(A)} \right)\le 0.
\label{eq:pressbd}
\end{equation}
Since the term $- \frac{q}{\mu(A)}$ doesn't affect finiteness, to obtain $P(\phi_A+q\H_A)<\infty$ it will be sufficient to show that $p_A(q) <q$ for sufficiently small $q$.

By the usual expansion of the pressure, see for example  \cite[Chapter 4]{ParPol90}, 
\begin{align*}
p_A(q) = P\left(\phi+q\left(1-\frac{1_A}{\mu(A)}\right)\right) &  = P(\phi)+ q \E_{\mu}\left(1-\frac{1_A}{\mu(A)}\right) + \frac{q^2D^2p_A(\hat q)}2 =  \frac{q^2D^2p_A(\hat q)}2 
\end{align*}
for some $\hat q\in (0, q)$.  Assuming that $q\le 1/2$, so likewise for $\hat q$, since  $|p_A(q)| \le |q|/\mu(A)$ for $q\in \mathbb{C}$, Cauchy's Integral Formula gives, for the contour $\gamma$ of radius $1$ centred at $\hat q$,
$$|D^2p_A(\hat q)|= \left|\frac{1}{\pi i}\oint_\gamma\frac{p_A(a)}{(a-\hat q)^3}~da\right| \le  2\int_0^1\frac{\left(\frac{|a|}{\mu(A)}\right)}{(1/2)^3}~da \le  \frac{16}{\mu(A)}.$$
So if $q\le  \mu(A)/16$, then $p_A(q)\le \frac q2$ and putting this into \eqref{eq:pressbd} we obtain
\begin{equation}
P\left(\phi_A+\frac{q}2\H_A \right) \le P\left(\phi_A+\H_A\left(q - p_A(q) \right) \right)\le  \frac{q}{\mu(A)}.
\label{eq:pressbd2}
\end{equation}

So setting $K':= \frac{1}{32 }$, we have the required finiteness for $q\le K'\mu(A)$. 

Finally we wish to find bounds on $P\left(\phi_A+q\H_A \right)$ for $q\in \mathbb{C}$ and $|q|$ sufficiently small.  If $Re(q)\ge 0$ then 
\eqref{eq:pressbd2} also implies that  if $Re(q) \le \frac{\mu(A)}{32}$ then $|P\left(\phi_A+q\H_A \right)| \le P\left(\phi_A+|q|\H_A \right)\le  2|q|/\mu(A)$.  On the other hand, if $Re(q)< 0$ we use the fact that if $q<0$ then by the Variational Principle (\cite[Theorem 3]{Sar99}),
$$P(\phi_A+q\H_A) \ge h_{\mu_{\phi_A}}+ \int \phi_A +q\H_A~d\mu_{\phi_A} = q\int \H_A~d\mu_{\phi_A} = \frac{q}{\mu(A)}.$$
Hence  $|P\left(\phi_A+q\H_A \right)| \le |P\left(\phi_A-|q|\H_A \right)|\le  |q|/\mu(A)$.  So we are finished if we set $K=K'$.
\end{proof}

\begin{cor} 
There is $C_P>0$ such that for $q<K \mu(A)/2$,
$$P\left(\phi_A+q\left(\H_A-\frac1{\mu_\phi(A)}\right)\right)<
 q^2\frac{C_P}{\mu(A)^2}.$$
\label{cor:var}
\end{cor}

\begin{proof}
Write $P_A(q):= P\left(\phi_A+q\left(\H_A-\frac{1}{\mu(A)}\right)\right)$. Proposition~\ref{prop:ana}  implies that this function is analytic for  $q<K\mu(A)/2$, so again by  \cite[Chapter 4]{ParPol90} there is a Taylor series expansion
\begin{align*}
P\left(\phi_A+q\left(\H_A-\frac{1}{\mu(A)}\right)\right) &  = P(\phi_A)+ q \E_{\mu_A}\left(\H_A-\frac{1}{\mu(A)}\right) + \frac{q^2D^2P_A(\hat q)}2 =\frac{q^2D^2P_A(\hat q)}2
\end{align*}
 for some $\hat q\in (0, q)$, where the last part follows by Proposition~\ref{prop:indprops}(e) and the Kac Lemma.
Using Cauchy's Integral Formula, where our contour $\gamma$ is of radius $K\mu(A)$, since Proposition~\ref{prop:ana} implies  $|P_A(a)| \le 2|a|/\mu(A)$ for $a\in \gamma$, we can write 
$$\left|D^2P_A(\hat q)\right|= \left|\frac{1}{\pi i}\oint_\gamma \frac{P_A(a)}{(a- \hat q)^{3}}~da\right|\le 2 K\mu(A) \frac{\left(\frac{2K\mu(A)}{\mu(A)}\right)}{\left(\frac K2\mu(A)\right)^3} = \frac{32}{ K\mu(A)^2}.$$
So we conclude by setting $C_P:= 16/K$.
\end{proof}

\subsection{Conformal measures and densities for perturbed potentials}

Suppose that $(A, F_A)$ are as in Proposition~\ref{prop:indprops}.  To save notation we will write our `distortion constant' as
$$C_D=C_D(\phi) := e^{\sum_{k=1}^nV_n(\phi)},$$
and note as in the proof of Proposition~\ref{prop:indprops}(b), $e^{\sum_{k=1}^nV_n(\phi_A)}\le C_D$.

\begin{lma}
Given $A\subset I$, for the range of $q$ for which $\mu_{\phi_A+q\H_A}$ exists, for 
$$\rho_{A, q}(x):= \frac{d\mu_{\phi_A+q\H_A}}{dm_{\phi_A+q\H_A}}(x),$$
then  $\frac1{C_D^2}\le \rho_{A,q}\le C_D^2$.
\label{lem:densbd} 
\end{lma}

\begin{proof}
As in for example \cite[Section 5]{MauUrb01},
$$\rho_{A, q}(x) = \lim_{n\to\infty} e^{-nP(\phi_A+q\H_A)} (\L_{\phi_A+q\H_A}^n 1)(x).$$
Suppose that $x, y\in A_i$ for $1\le i\le p$ and observe that if $P\in \P_n^A$ has $F_A^nP=A_i$ then there are exactly two points $x_P, y_P\in P$ such that $F_A^n(x_P)=x$ and $F_A^n(y_P)=y$.  Then
\begin{align*}
\frac{\rho_{A, q}(x)}{\rho_{A, q}(y)} & =\frac{\lim_{n\to \infty} e^{-nP(\phi_A+q\H_A)} \sum_{P\in \P_n^A}e^{S_n\left(\phi_A+q\H_A\right)(x_P)}}{\lim_{n\to \infty} e^{-nP(\phi_A+q\H_A)} \sum_{P\in \P_n^A}e^{S_n\left(\phi_A+q\H_A\right)(y_P)}}
= \lim_{n\to \infty} \frac{ \sum_{P}e^{S_n\left(\phi_A+q\H_A\right)(x_P)}}{\sum_{P}e^{S_n\left(\phi_A+q\H_A\right)(y_P)}}\\ 
&= 1 + \lim_{n\to \infty} \frac{ \sum_{P}e^{S_n\left(\phi_A+q\H_A\right)(y_P)}\left(e^{S_n\left(\phi_A+q\H_A\right)(x_P)- S_n\left(\phi_A+q\H_A\right)(y_P)}-1\right)}{\sum_{P}e^{S_n\left(\phi_A+q\H_A\right)(y_P)}}\\
&=1+ \left(C_D^{\pm}-1\right)= C_D^{\pm}.
\end{align*}
  
 It remains to prove that $\rho_{A, q}$ doesn't vary too much between the $A_i$'s.  Given $1\le i, j\le p$, if $x\in A_i$ then there is some $b\ge 1$ such that $F_A^b(x) \in A_j$.  Let  $U\ni x$ be an interval in $A_i$ such that $F_A: U\to F_A^b(U)\subset A_j$ is a bijection.  Then we estimate $\rho_{A, q}$ on $A_j$ by 
\begin{align*}
\frac{\mu_{A, q}(F_A^n(U))}{m_{A, q}(F_A^n(U))} &= \frac{\int \L_{A, q}^b (1_U\cdot \rho_{A, q})~dm_{A, q}}{\int_U e^{S_b \left(\phi_A+q\H_A\right)-bP}~dm_{A, q}}= \frac{\int_U e^{S_b \left(\phi_A+q\H_A\right)-bP}\rho_{A, q}~dm_{A, q}}{\int_U e^{S_b \left(\phi_A+q\H_A\right)-bP}~dm_{A, q}}= C_D^{\pm}\rho_{A, q}(x).
\end{align*}
 This implies that $\frac1{C_D^2}\le \rho_{A, q}\le C_D^2$.
\end{proof}

Now let $\Q^A_n=\cup_i \Q^{A, i}_n$ be the partition of $I$, up to sets of measure zero, such that for each $Q\in \Q^{A, i}_n$, $\H_A^n(x)$ is constant on $Q$ and $f^{\H_A^n|_Q}(Q) =A_i$.  Let $\P^{A, i}_n\subset \P_n^A$ be the elements of $ \Q^{A, i}_n$  in $A$.  We extend $\phi_A$ outside $A$ so that if $x\in Q\in  \Q^{A}_1\sm \P_1^A$ then $\phi_A(x) = S_{\H_A(x)}\phi(x)$.

In the following result we will apply theory from \cite{JurTod25} which comes from systems with holes, hence the notation $C_H$ below.

\begin{prop}
For all $0\le q\le K\mu(A)$ there is $m_{\phi_A+q\H_A}$, a $\left(\phi_A+q\H_A-P\left(\phi_A+q\H_A\right)\right)$-conformal measure for the induced system on $A$.  Moreover, this extends to a measure outside $A$ such that if $Q\in \Q^{A, i}_n$ and $U\subset Q$ is an interval, then 
\begin{equation}
m_{\phi_A+q\H_A}(f^{\H_A^n}(U)) = \int_Ue^{-S_n\left(\phi_A+q\H_A\right) + nP\left(\phi_A+q\H_A\right)} dm_{\phi_A+q\H_A}.
\label{eq:confex}
\end{equation}
There also exist $C_m\ge 1$ and $C_H>0$ depending only on $(I, f, \phi)$ and $\lambda_A\in (0, 1-C_H\mu(A))$ such that for all $1\le i\le p$ and any $\n\in \N$,  if $\lambda_Ae^q<1$ then
$$\sum_{Q\in \Q^{A, i}_n\sm \P_n^{A, i}} m_{\phi_A+q\H_A}(Q)\le \frac{C_m}{1-\lambda_Ae^q}  \sum_{P\in \P^{A, i}_{n-1}}m_{\phi_A+q\H_A}(P).$$
\label{prop:conf}
\end{prop}

We will choose the normalisation of these measures to be such that $m_{\phi_A+q\H_A}(A)=1$.  For $q=0$ this means that as in Propsition~\ref{prop:indprops},  $m_{\phi_A} = m_\phi/m_\phi(A)$.   Note the final condition in the proposition, $\lambda_Ae^q<1$, follows if $q<C_H\mu(A)$.

\begin{proof}
First note that by Proposition~\ref{prop:ana}, $P\left(\phi_A+q\H_A\right)<\infty$ and the existence of 
$m_{\phi_A+q\H_A}$ on $A$ follows for example by \cite[Theorem 4]{Sar99}.  The extension outside $A$ is defined via \eqref{eq:confex}.

To estimate the total mass of the extended measure we use results of \cite[Section 4]{JurTod25}.  In particular we define the punctured operator $\L_{\phi, A}: \mathfrak{B} \to \mathfrak{B}$ by 
$$(\L_{\phi, A}g)(x) = \sum_{f(y)= x, y\notin A} e^{\phi(y)}g(y)$$
 with leading eigenvalue $\lambda_A$ where $\lambda_A\in (0, 1-C_H\mu(A))$.   Note that this implies that for any $x\in I$,
 \begin{equation}
 (\L_{\phi+q, A}1)(x) =  \sum_{f(y)= x, y\notin A} e^{\phi(y)+q}\le \lambda_A e^q.
 \label{eq:perttrans}
 \end{equation}
 
Given $i$ and $Q\in \Q^{A, i}_n\sm \P_n^{A, i}$, there is a unique $P\in \P^{A, i}_{n-1}$ such that $f^{ \H_A|_Q}(Q) = P$.  We can also write
 \begin{align*} m_{\phi_A+q\H_A}(Q)&= m_{\phi_A+q\H_A}(P)e^{(\phi_A+q\H_A)(x)-P(\phi_A+q\H_A)} \\
 &\le C_D m_{\phi_A+q\H_A}(P)e^{S_{\H_A}(\phi+q)(x)} 
\end{align*}
 for any $x\in Q$.  Summing over all such $Q$ (for a given $P$) gives a sum bounded by
 $$C_D^{2}\sum_{j\ge 1} \int_P \L_{\phi+q, A}^j1~dm_{\phi_A+q\H_A}\le C_D^{2}\sum_{j\ge 1} \left(\lambda_A e^q\right)^jm_{\phi_A+q\H_A}(P)$$
 by \eqref{eq:perttrans}.  So if $\lambda_Ae^q<1$, this is uniformly bounded, indeed we deduce that there is $C_m\ge 1$ such that for each $i$,
 $$\sum_{Q\in \Q^{A, i}_n\sm \P_n^{A, i}}m_{\phi_A+q\H_A}(Q)\le \frac{C_m}{1-\lambda_Ae^q} \sum_{P\in \P^{A, i}_{n-1}}m_{\phi_A+q\H_A}(P),$$
 as required.
\end{proof}

\subsection{Proof of the large deviation result}

\begin{proof}[Proof of Proposition~\ref{prop:ChaLep}]
We start by adjusting $K$ if necessary, so that $K\le C_H$ for $C_H$ as in Proposition~\ref{prop:conf}.  Throughout the proof we assume $0\le q\le K\mu(A)$ and that $\mu(A)<1/C_H$.

Let $J_{n, i}^+(\eta)$ denote the collection of elements $Q\in \Q^{A, i}_n$ containing some $x$ such that $S_n\H_A(x)>n\left(\frac{1}{\mu(A)}+\eta\right)$: since $S_n\H_A$ is constant on these cylinders, this will in fact be true for any $x\in Q\in J_{n, i}^+(\eta)$, so  $$S_n\H_A(x) -n\eta- n\frac{1}{\mu(A)}= S_n\left(\H_A(x) -\eta- \frac{1}{\mu(A)}\right)>0.$$

 Then since $m_{\phi_A}$ is the normalised version of $m_{\phi}$, 
 for $q>0$, by \eqref{eq:conf} and Proposition~\ref{prop:conf} (and letting $x_Q$ be an arbitrary point in a given set $Q$), 
  \begin{align*}
m_{\phi}\left(S_n\H_A>n\left(\frac{1}{\mu(A)}+\eta\right)\right) &  \le {m_\phi(A)}\sum_{i}\sum_{Q\in J_{n, i}^+(\eta)}m_{\phi_A}(Q)\\
&
\hspace{-3.5cm} \le C_D{m_\phi(A)}\sum_{i}m_{\phi_A}(A_i)\sum_{Q\in J_{n, i}^+(\eta)}e^{S_n\phi_A(x_Q)}\\
 &\hspace{-3.5cm} \le  C_D{m_\phi(A)}\sum_{i}m_{\phi_A}(A_i)\sum_{Q\in J_{n, i}^+(\eta)}e^{q\left(S_n\left(\H_A-\eta-\frac{1}{\mu(A)}\right)(x_Q)\right)+S_n\phi_A(x_Q)}\\
    &\hspace{-3.5cm} \le C_D {m_\phi(A)}  e^{-nq\left(\eta+\frac{1}{\mu(A)}\right)} \sum_{i}m_{\phi_A}(A_i)
    \sum_{Q\in \Q^{A, i}_n}e^{S_n\left(\phi_A+q\H_A\right)(x_Q)}\\
  & \hspace{-3.5cm} \le C_D^2{m_\phi(A)} e^{n\left(P(\phi_A+q\H_A)-\frac{q}{\mu(A)}-q\eta\right)}\sum_{i} \frac{m_{\phi_A}(A_i)}{m_{\phi_A+q\H_A}(A_i)} \sum_{Q\in  \Q^{A,i}_n}m_{\phi_A+q\H_A}(Q) \\ 
    & \hspace{-3.5cm} \le  \frac{C_mC_D^2{m_\phi(A)}}{1-\lambda_Ae^q} e^{n\left(P(\phi_A+q\H_A)-\frac{q}{\mu(A)}-q\eta\right)}\sum_{i} \frac{m_{\phi_A}(A_i)}{m_{\phi_A+q\H_A}(A_i)} \sum_{P\in  \P^{A,i}_n\cup \P^{A, i}_{n-1}}m_{\phi_A+q\H_A}(P) \\ 
   & \hspace{-3.5cm} \le  \frac{C_mC_D^6{m_\phi(A)}}{1-\lambda_Ae^q}  e^{n\left(P(\phi_A+q\H_A)-\frac{q}{\mu(A)}-q\eta\right)}\\
   &\hspace{0cm} \sum_{i} \frac{m_{\phi_A}(A_i)}{\mu_{\phi_A+q\H_A}(A_i)} \left(\mu_{\phi_A+q\H_A}(F_A^{-n}A_i)+ \mu_{\phi_A+q\H_A}(F_A^{-(n-1)}A_i)\right) \\ 
   &\hspace{-3.5cm}\le \frac{2C_mC_D^6{\mu(A)}}{1-\lambda_Ae^q}  e^{n\left(P\left(\phi_A+q\left(\H_A-\frac{1}{\mu(A)}\right)\right)-q\eta\right)} \end{align*}
where in the sixth inequality we used Proposition~\ref{prop:conf} (since  $q\le C_H\mu(A)$ so $\lambda_Ae^q<1$) and in the seventh we used Lemma~\ref{lem:densbd}. 

We now estimate $P\left(\phi_A+q\left(\H_A-\frac{1}{\mu(A)}\right)\right)-q\eta$.
Let $\eta<\frac{KC_P}{\mu(A)}$ and set $q=\frac{\eta \mu(A)^2}{2C_P}<K\mu(A)/2$ where $C_P$ is as in Corollary~\ref{cor:var}.  Then Corollary~\ref{cor:var} implies 
\begin{align*}
P\left(\phi_A+q\left(\H_A-\frac{1}{\mu(A)}\right)\right)-q \eta& <q\left(-\eta+ q\frac{C_P}{\mu(A)^2}\right)=-\frac{\eta^2 \mu(A)^2}{2C_P}.
\end{align*}

To find a uniform bound on $\frac{\mu(A)}{1-\lambda_Ae^q}\le \frac{\mu(A)}{1-(1-C_H\mu(A))e^q}$, note that if $q\le C_H\mu(A)/2$, then 
$$1-(1-C_H\mu(A))e^q\ge  \frac{C_H\mu(A)}3 $$
for $\mu(A)<1/C_H$.   So recalling that $K\le C_H$, this means that $\frac{\mu(A)}{1-\lambda_Ae^q}\le \frac3{C_H}$.

Finally,
\begin{align*}
\mu_\phi\left(S_n\H_A>n\left(\frac1{\mu(A)}+\eta\right)\right) &\leq  C_D^2 m_\phi\left(S_n\H_A>n\left(\frac1{\mu(A)}+\eta\right)\right),
\end{align*}
which completes the proof for $c_1=KC_P$, $c_2= \frac1{2C_P}$ and $c_3 = \frac6{C_H}C_mC_D^8$.
\end{proof}

\section{Multifractal proofs}\label{mfsection}

In this section we prove Proposition \ref{prop:mftoolintro}. Namely we show that an $(\epsilon, \delta)$-multifractal discretisation $\A_{\delta,\epsilon}$ can be constructed for which $\sum_ {A\in \A_{\delta,\epsilon}} e^{-\mathbf{C}_\delta\mu(A)}$ can be bounded uniformly over all small scales $\delta$ (for $\mathbf{C}_\delta=p\delta^{-D}\log(1/\delta)$ with an appropriate choice of $p$).

Before getting to that proof, we define the
\emph{coarse multifractal spectrum of $\mu$}
\[
\mathcal{F}_\mu(\alpha)
   := \lim_{r \to 0}\,
     \limsup_{\delta \to 0}
     \frac{\log N_\delta(\alpha,r)}{-\log \delta},
\]
where $N_\delta(\alpha,r)$ denotes the number of boxes in the standard
$\delta$-grid $G_\delta$ whose $\mu$-mass lies in the range $[\delta^\alpha,\delta^{\alpha-r}] $:
\[
N_\delta(\alpha,r)
   := \#\bigl\{ G\in G_\delta :
        \mu(G)\in[\delta^\alpha,\delta^{\alpha-r}] \bigr\}.
\]

Note that the constants $d$ and $D$ defined in \eqref{minkowski} satisfy
\[
d=\inf\{\alpha\ge0 : \mathcal{F}_\mu(\alpha)>0\} \;\; \;\textnormal{and} \;\;\; D=\sup\{\alpha\ge0 : \mathcal{F}_\mu(\alpha)>0\}.
\]

\begin{lma}\label{mfacip}
$\mathcal{F}_\mu(D)=1$ if and only if $\mu$ is the absolutely continuous
invariant probability measure (acip) for $f$.
\end{lma}

\begin{proof}
If $\mu=\mu_\phi$ is an acip, clearly $\mathcal{F}_\mu(D)=1$. If $\mu$ is not an acip, then general multifractal analysis theory of Gibbs measures for Holder continuous potentials (e.g.\  \cite[Theorem 6.2.1]{barreira}) implies that $\mathcal{F}_\mu$ is analytic and strictly convex on $[d,D]$. Moreover, $\mathcal{F}_\mu$ is the Legendre transform of an analytic and strictly concave function $\beta$ which is defined implicitly by
$$P(-\beta(q)\log|f'|+q\phi)=0.$$
In particular, for all $q \in \R$,
$$F_\mu(-\beta'(q))=\beta(q)-q\beta'(q).$$
By standard theory, $F_\mu(-\beta'(0))=1$ and $\lim_{q \to -\infty} -\beta'(q)=D$. Thus strict concavity of $\beta$ implies that $-\beta'(0)<D$. In particular $\mathcal{F}_\mu(D)<1$.
\end{proof}

We now deal with the setting where $\mu$ is an acip, where the construction of $\A_{\delta,\epsilon}$ is fairly straightforward, and then use $\F_\mu$ in the other cases. Given $\rho>0$ we say that $\mathscr{P}_{\rho}$ is a \emph{canonical partition} of $I$ into cylinders of diameter roughly $\rho$, if
$$\mathscr{P}_{\rho}:=\{ \Pi[\i]: |\Pi[\i]| <\rho \leq |\Pi[\i^{-}]|\},$$
where $\i^{-}$ denotes $\i$ with the last symbol removed. Note that since our map has only finitely many uniformly expanding branches, for any $\Pi[\i] \in \mathscr{P}_{\rho}$,
$$|\Pi[\i]| \asymp \rho$$
where the implied constants are independent of both $\rho$ and $\i$.

\begin{prop} Suppose $\mu$ is an acip. For all $\epsilon, \delta \in (0,1)$ we can construct an $(\epsilon,\delta)$-multifractal discretisation of $\mu$ such that
$$\# \A_{\delta,\epsilon}=O\left(\frac{1}{\delta\epsilon}\right)$$
and each $A \in\A_{\delta,\epsilon}$ is made up of a finite union of cylinders.
 \label{acipprop}
\end{prop}
 Observe that this implies Proposition \ref{prop:mftoolintro} in the setting where $\mu$ is an acip.

\begin{proof}
Since $\mu$ is an acip, Lemma \ref{lem:weakreg} implies that there exist constants $a<b$ such that for all $\delta>0$ and $x \in I$, 
\begin{equation} \label{acip}
a\delta \leq \mu(B(x,\delta)) \leq b\delta.
\end{equation} 

Choose $\rho=a/b$. For each $\delta, \epsilon \in (0,1)$ consider the canonical partition $\mathscr{P}_{\rho\epsilon \delta}$ of $I$  into cylinders of roughly diameter $\rho\epsilon\delta$.

Define $\A_{\delta,\epsilon}$ to consist of all intervals $J$ which can be written as a union of cylinders from $\mathscr{P}_{\rho \epsilon \delta}$ and which have length $|J| \in [2\delta-2\delta \rho\epsilon, 2\delta]$. 

To see this is an $(\epsilon,\delta)$-multifractal discretisation of $\mu$, consider a ball $B(x,\delta)$ and let $U$ be the largest union of cylinders in $\mathscr{P}_{\rho \epsilon \delta}$ that is contained in $B(x,\delta)$. This implies $|U|>2\delta-2\rho \epsilon\delta$, hence $U=A$ for some $A \in \A_{\delta,\epsilon}$. Moreover,
$$\frac{\mu(A)}{\mu(B(x,\delta))}=1-\frac{\mu(B(x,\delta)\setminus A)}{\mu(B(x,\delta))} \geq 1- \frac{b\rho \epsilon}a=1-\epsilon$$
since $\mu(B(x,\delta)) \geq a\delta$ by \eqref{acip} and similarly $\mu(B(x,\delta)\setminus A) \leq b\rho \epsilon\delta$ since $B(x,\delta)\setminus A$ is made up of two intervals, each of which has diameter at most $\rho \epsilon \delta$.

To verify $\# \A_{\delta,\epsilon}=O\left(\frac{1}{\delta\epsilon}\right)$, note that for each left end point of an interval in $\mathscr{P}_{\rho \epsilon \delta}$, there are $O(1)$ elements in $\A_{\delta,\epsilon}$ that have this as their left end point (this is because we have a uniform upper bound on the derivative $f'$ of our map). Therefore since $\#\mathscr{P}_{\rho \epsilon \delta}=O(1/\epsilon\delta)$ the result follows.

Finally, this proves Proposition \ref{prop:mftoolintro} in the setting where $\mu$ is an acip because for $p\ge 1/a$, setting $\mathbf{C}_\delta=p\delta^{-d}\log(1/\delta)$ we get
$$\sum_{A \in \A_{\delta,\epsilon}} e^{-\mathbf{C}_\delta \mu(A)} \leq \#\A_{\delta,\epsilon} \cdot \delta=O(1/\epsilon).$$
\end{proof}

Next we turn to the more general setting where $\mu$ is a multifractal measure. 
Here the construction of $\A_{\delta,\epsilon}$ is more delicate. The difficulty is
that $\mu$ may have ``spikes'': very small intervals that carry comparatively
large mass. In such regions, the measure of a ball $B(x,\delta)$ is no longer
well-approximated by a simple power of $\delta$, and the local mass distribution
can vary by several orders of magnitude as $x$ varies.

To see why this causes an issue, let us demonstrate what happens if we apply a similar strategy to the one above. Namely, let us try to partition $I$ into a canonical partition $\mathscr{P}_{\rho(\epsilon,\delta)}$ which is fine enough that any cylinder in $\mathscr{P}_{\rho(\epsilon,\delta)}$ takes up at most $\epsilon$ of any $\delta$-ball which contains it. The information given by the endpoints of the multifractal spectrum $\F_\mu(\alpha)$ tell us that we require
$$\frac{\rho(\epsilon,\delta)^{d}}{\delta^{D}}<\epsilon$$
which can be achieved provided $\rho(\epsilon,\delta)<\epsilon^{\frac{1}{d}}\delta^{\frac{D}{d}}$. But applying this directly gives, again with $p\ge 1/a$, 
\begin{equation}\sum_{A \in \A_{\delta,\epsilon}} e^{-\mathbf{C}_\delta \mu(A)} \leq \#\A_{\delta,\epsilon} \cdot \delta=O(\epsilon^{-\frac{1}{d}}\delta^{1-\frac{D}{d}})\label{rhoprob}\end{equation}
which introduces a scale-dependent error to our estimates, whenever $d<D$.

Therefore more subtle arguments are required to construct $\A_{\delta,\epsilon}$. Roughly speaking, by repeating the general strategy above, a multifractal discretisation $\A_{\delta,\epsilon}^1$ will be constructed only on parts of the space where the measure of balls at scale $\delta$ are exponentially larger that $\delta^{D}$. The exponential gain $\mu(A) \gg\delta^{D}$ for $A \in\A_{\delta,\epsilon}^1$ will compensate for the problematically large cardinality $\#\A_{\delta,\epsilon}^1$, restoring uniform boundedness over $\delta>0$ in \eqref{rhoprob}. However, in parts of the space where the measure of balls at scale $\delta$ is closer to the global minimum $\delta^{D}$, it will be possible to construct a multifractal discretisation $\A_{\delta,\epsilon}^2$ of much smaller cardinality, since the multifractal analysis of $\mu$ implies that this part of the space is a lot smaller and the measure must necessarily be more evenly distributed there.

\begin{prop}\label{mainmf} Let $\mu$ be any measure which is not an acip.  Letting  $\mathbf{C}_\delta= p\delta^{-D}\log(1/\delta)$ for an appropriate choice of $p$, the following holds. 
For all sufficiently small $\epsilon, \delta \in (0,1)$ and sufficiently small
$$0<r<\frac{d}{2}(1-\F_\mu(D))$$
we can construct an $(\epsilon,\delta)$-multifractal discretisation of $\mu$ such that  $\A_{\delta,\epsilon}=\A_{\delta,\epsilon}^1\cup \A_{\delta,\epsilon}^2$ and 
\begin{align*}
\sum_{A \in \A^1_{\delta,\epsilon}} e^{-\mathbf{C}_\delta \mu(A)}& \leq \epsilon^{-\frac{1}{d}}\delta^{-\frac{D}{d}}\delta^{\delta^{-r}/2}, \\
\sum_{A \in \A^2_{\delta,\epsilon}} e^{-\mathbf{C}_\delta \mu(A)} &\leq\epsilon^{-\frac{1}{d}} \delta^{\frac{1}{2}-\frac{r}{d}-\frac{\F_\mu(D)}{2}}.
\end{align*}
Moreover each $A \in\A_{\delta,\epsilon}$ is made up of a finite union of cylinders.
\end{prop}

\begin{proof}
By Lemma \ref{mfacip} we know that the multifractal spectrum $\F_\mu(D)<1$. First we describe how to construct $ \A^1_{\delta,\epsilon}$. For now, $r>0$ can be arbitrary. Let us denote 
$$\rho(\epsilon,\delta):=\epsilon^{\frac{1}{d}}\delta^{\frac{D}{d}}$$ and let $\mathscr{P}_{\rho(\epsilon,\delta)}$ denote the canonical partition. Let $\mathcal{B}_{\rho(\epsilon,\delta)}$ consist of all intervals $J$ which can be written as a union of cylinders from $\mathscr{P}_{\rho(\epsilon,\delta)}$ and which have length $|J| \in [2\delta-2\rho(\epsilon,\delta), 2\delta]$, i.e., 
$$\mathcal{B}_{\rho(\epsilon,\delta)}:= \left\{ J \text{ an interval}: J= C_1\cup\cdots \cup C_k \text{ and } C_i\in \mathscr{P}_{\rho(\epsilon,\delta)} \text{ for } i=1, \ldots, k\right\}.$$
 We define
$$\A^1_{\delta,\epsilon,r}:=\{J \in \mathcal{B}_{\rho(\epsilon,\delta)}: \mu(J) \geq \delta^{D-r}/2\}.$$
Then 
$$\#\A^1_{\delta,\epsilon,r} \lesssim \rho(\epsilon,\delta)^{-1}=\epsilon^{-\frac{1}{d}}\delta^{-\frac{D}{d}}$$
and for any $A \in \A^1_{\delta,\epsilon,r}$ we have $\mu(A)\delta^{-D} \geq \delta^{-r}/2$ therefore for $p \geq 1$,
$$\sum_{A \in \A^1_{\delta,\epsilon,r}} e^{-\mathbf{C}_\delta \mu(A)} \leq \epsilon^{-\frac{1}{d}}\delta^{-\frac{D}{d}}\delta^{\delta^{-r}/2}.$$
This demonstrates the first inequality in Proposition \ref{mainmf}, when we later define $\A^1_{\delta,\epsilon}$ by $\A^1_{\delta,\epsilon}=\A^1_{\delta,\epsilon,r}$ for an appropriate choice of $r>0$. 

Next, we describe how to construct $ \A^2_{\delta,\epsilon}$. Since $\F_\mu(D)<1$ we can fix 
$$0<r<d\left(\frac{1}{2}-\frac{\F_\mu(D)}{2}\right)$$
sufficiently small such that 
$$\#\mathscr{P}_{\delta/2}^r \leq \delta^{-\frac{1}{2}-\frac{\F_\mu(D)}{2}}$$
 where
$$\mathscr{P}_{\delta/2}^r:=\left\{C \in \mathscr{P}_{\delta/2}: \mu(C) < \delta^{D-r}\right\}$$
and $\mathscr{P}_{\delta/2}$ is the canonical partition. Note that since $\F_\mu(D)<1$, the exponent in the upper estimate on $\#\mathscr{P}_{\delta/2}^r$ satisfies the bounds $\F_\mu(D)< \frac{1}{2}+\frac{\F_\mu(D)}{2}<1.$

Consider an interval $U \in\mathscr{P}_{\delta/2} \setminus\mathscr{P}_{\delta/2}^r$ which neighbours an interval $V \in\mathscr{P}_{\delta/2}^r$. Let $\mathbf{k}$ be the shortest word for which $\Pi[\mathbf{k}] \subset U$, the interval $\Pi[\mathbf{k}]$  neighbours $V$ and  $\mu(\Pi[\mathbf{k}])<\delta^{D-r}$, see Figure \ref{fig:k} below. Let $\mathcal{Q}_\delta^r$ consist of all $\mathbf{k}\in \Sigma^*$ which can be obtained in this way for some $U \in\mathscr{P}_{\delta/2} \setminus\mathscr{P}_{\delta/2}^r$ and $V \in\mathscr{P}_{\delta/2}^r$. 

\begin{figure}[h]
\begin{tikzpicture}[thick, scale=0.5]
\draw[ - ] (0,2) -- (20,2);
\draw[ - ] (2,1.7) -- (2, 2);
\draw[ - ] (3,1.7) -- (3, 2);
\draw[ - ] (5,1.7) -- (5, 2);
\draw[ - ] (8,1.7) -- (8, 2);
\draw[ - ] (12,1.7) -- (12, 2);
\draw[ - ] (13.5,1.7) -- (13.5, 2);
\draw[ - ] (15.5,1.7) -- (15.5, 2);
\draw[ - ] (19,1.7) -- (19, 2);
\draw[ - ] (0.2,1.7) -- (0.2, 2);

\draw[thin, blue, - ] (2,2.5) -- (3, 2.5);
\draw[ thin, blue, - ] (8,2.5) -- (12, 2.5);
\draw[thin,  blue, - ] (13.5,2.5) -- (15.5, 2.5);
\draw[thick, red, - ] (6.7,2.5) -- (8, 2.5);

\draw (7,1.5) node[below] {{\small $U$}};
\draw (10,1.5) node[below] {{\small $V$}};

\draw (7.2,2.5) node[above] {{\small $\Pi[\k]$}};

\draw (17,1.2) node[below] {{\small $\mathscr{P}_{\delta/2}$}};

\draw (10,5) node[above] {{\small $\mathscr{P}_{\delta/2}^r$}};

\draw[blue, ->] (9.5, 5)--(3, 2.9);

\draw[blue, ->] (10, 5)--(10, 2.9);

\draw[blue, ->] (10.5, 5)--(14.5, 2.9);
\end{tikzpicture}
  \caption{The long black line here is a segment of $[0, 1]$ and the main partition of this segment represents $\mathscr{P}_{\delta/2}$, and the thin blue intervals belong to $\mathscr{P}_{\delta/2}^r$. For the indicated neighbouring intervals $U\in \mathscr{P}_{\delta/2}\sm \mathscr{P}_{\delta/2}^r$ and $V\in \mathscr{P}_{\delta/2}^r$, the thick red interval $\Pi[\k]$ is chosen as the cylinder of shortest symbolic length contained in $U$, adjacent to $V$ and whose measure is at most $\delta^{D-r}$.}
    \label{fig:k}
\end{figure}
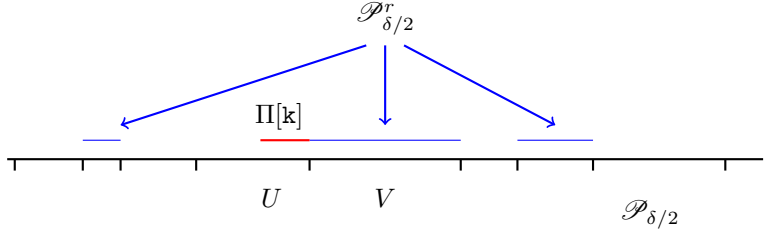

Define 
$$\kappa(\epsilon,\delta):=\epsilon^{\frac{1}{d}}\delta^{\frac{r}{d}}$$
and let $\mathscr{P}_{\kappa(\epsilon,\delta)}$ be the canonical partition.

Now define $\A^2_{\delta,\epsilon}$ to be all intervals $J$ which can be written as a union of cylinders from  
\begin{equation}\label{weirdset}
 \left\{\Pi[\i\j]: \Pi[\i] \in \mathscr{P}_{\delta/2}^r, \;\Pi[\j ]\in \mathscr{P}_{\kappa(\epsilon,\delta)}\right\} \cup \left\{\Pi[\k^-\j]: \mathbf{k} \in \mathcal{Q}_{\delta}^r,\; \Pi[\j] \in \mathscr{P}_{\kappa(\epsilon,\delta)}\right\}
\end{equation}
and which have length $|J| \in [2\delta-2\rho(\epsilon,\delta), 2\delta]$. Since $\#\mathscr{P}_{\delta/2}^r \leq\delta^{-\frac{1}{2}-\frac{\F_\mu(D)}{2}}$,
$$\#\A^2_{\delta,\epsilon} \lesssim\delta^{-\frac{1}{2}-\frac{\F_\mu(D)}{2}}\kappa(\epsilon,\delta)^{-1}=\epsilon^{-\frac{1}{d}} \delta^{-\frac{r}{d}-\frac{1}{2}-\frac{\F_\mu(D)}{2}}.$$
In particular, for $p\ge 1/a$, 
$$\sum_{A \in \A^2_{\delta,\epsilon}} e^{-\mathbf{C}_\delta \mu(A)} \leq\delta \cdot\epsilon^{-\frac{1}{d}} \delta^{-\frac{r}{d}-\frac{1}{2}-\frac{\F_\mu(D)}{2}}=\epsilon^{-\frac{1}{d}} \delta^{\frac{1}{2}-\frac{r}{d}-\frac{\F_\mu(D)}{2}}.$$

We define $\A^1_{\delta,\epsilon}=\A^1_{\delta,\epsilon,r}$ for the choice of $r$ which was fixed in the construction of $\A^2_{\delta,\epsilon}$ and set $\A_{\delta,\epsilon}=\A^1_{\delta,\epsilon} \cup \A^2_{\delta,\epsilon}$. It remains to show that this is a multifractal discretisation of $\mu$.

To see this, let us fix an arbitrary ball $B=B(x,r)$ and assume $\epsilon<1/2$. If $\mu(B)>\delta^{D-r}$, then the best approximation of $B$ by $A \in \mathcal{B}_{\rho(\epsilon,\delta)}$ with $B \subset A$ will have measure
$$\mu(A) \geq \mu(B)-\rho(\epsilon,\delta)^{d}>\delta^{D-r}-\epsilon \delta^{D} \ge \delta^{D-r}/2,$$ which implies $A \in \A^1_{\delta,\epsilon}=\A^1_{\delta,\epsilon,r}$. Moreover, $B \setminus A$ is contained in at most two intervals from $\mathscr{P}_{\rho(\epsilon,\delta)}$, hence
$$\frac{\mu(B \setminus A)}{\mu(B)} \lesssim\frac{\rho(\epsilon,\delta)^{d}}{\delta^{D}}=\epsilon.$$

On the other hand, if $\mu(B)\leq\delta^{D-r}$ then $B$ necessarily contains some interval $C \in \mathscr{P}_{\delta/2}^r$ and moreover, $B \cap(\mathscr{P}_{\delta/2}\setminus \mathscr{P}_{\delta/2}^r) \subset \bigcup_{\k \in\mathcal{Q}_\delta^r} \Pi[\k]$. Hence we can find $A \in  \A^2_{\delta,\epsilon}$ such that $A \subset B$. Then $B \setminus A$ will be contained in at most two intervals from \eqref{weirdset}, and applying the quasi-Bernoulli property to any set in \eqref{weirdset} yields 
$$\mu(B \setminus A) \lesssim \delta^{D-r} \cdot \kappa(\epsilon,\delta)^{d}=\epsilon\delta^{D}.$$
Therefore in this case we also have
$$\frac{\mu(B \setminus A)}{\mu(B)} \lesssim\frac{\epsilon\delta^{D}}{\delta^{D}}=\epsilon.$$

\end{proof}

We are now ready to prove Proposition \ref{prop:mftoolintro} in the general setting.

\begin{proof}[Proof of Proposition \ref{prop:mftoolintro}]
If $\mu$ is an acip, then the proof follows from Proposition \ref{acipprop}. Suppose $\mu$ is not an acip. Then by Proposition \ref{mainmf}, for all sufficiently small $\epsilon, \delta \in (0,1)$ and sufficiently small
$$0<r<\frac{d}{2}(1-\F_\mu(D))$$
we can construct an $(\epsilon,\delta)$-multifractal discretisation of $\mu$ such that  $\A_{\delta,\epsilon}=\A_{\delta,\epsilon}^1\cup \A_{\delta,\epsilon}^2$ and 
\begin{align*}
\sum_{A \in \A^1_{\delta,\epsilon}} e^{-\mathbf{C}_\delta \mu(A)}& \leq \epsilon^{-\frac{1}{d}}\delta^{-\frac{D}{d}}\delta^{\delta^{-r}/2}\\
\sum_{A \in \A^2_{\delta,\epsilon}} e^{-\mathbf{C}_\delta \mu(A)} &\leq\epsilon^{-\frac{1}{d}} \delta^{\frac{1}{2}-\frac{r}{d}-\frac{\F_\mu(D)}{2}}.
\end{align*}
In particular
$$\sum_{A \in \A_{\delta,\epsilon}} e^{-\mathbf{C}_\delta \mu(A)} \leq \sum_{A \in \A^1_{\delta,\epsilon}} e^{-\mathbf{C}_\delta \mu(A)}+\sum_{A \in \A^2_{\delta,\epsilon}} e^{-\mathbf{C}_\delta \mu(A)} \leq\epsilon^{-\frac{1}{d}}\delta^{-\frac{D}{d}}\delta^{\delta^{-r}/2}+\epsilon^{-\frac{1}{d}} \delta^{\frac{1}{2}-\frac{r}{d}-\frac{\F_\mu(D)}{2}}.$$
This can be easily seen to be of order $O(\epsilon^{-1/d})$ since $\frac{1}{2}-\frac{r}{d}-\frac{\F_\mu(D)}{2}>0$.
\end{proof}

\end{document}